\documentclass[a4paper]{amsart}

\usepackage{color}
\usepackage[dvipdfm]{hyperref}
\usepackage{amssymb}
\usepackage{amsthm}
\usepackage{amscd}
\usepackage{amsmath}
\usepackage[all]{xy}
\usepackage{graphicx}
\theoremstyle{plain}
\newtheorem{theorem}{Theorem}
\newtheorem*{theorem*}{Theorem}

\newtheorem*{MultLemma}{Multiplicativity Lemma}
\newtheorem*{maintheorem*}{Main Theorem}
\newtheorem{proposition}[theorem]{Proposition}
\newtheorem{corollary}[theorem]{Corollary}
\newtheorem{lemma}[theorem]{Lemma}
\newtheorem{claim}[theorem]{Claim}
\newtheorem*{conjecture*}{Conjecture}

\theoremstyle{definition}
\newtheorem{definition}{Definition}
\newtheorem*{definition*}{Definition}
\newtheorem{example}{Example}
\newtheorem*{example*}{Example}
\newtheorem*{notation*}{Notation}
\newtheorem*{notation-conv*}{Notation and convention}
\newtheorem*{convention*}{Convention}
\newtheorem*{hypothesis*}{Hypothesis}

\theoremstyle{remark}
\newtheorem{remark}{Remark}

\newcommand{\ie}{\emph{i.e.}}
\newcommand{\cf}{\emph{cf.}\:}

\newcommand{\mf}[1]{\mathfrak{#1}}

\newcommand{\myleq}{\leqslant}
\newcommand{\mygeq}{\geqslant}

\newcommand{\ZZ}{{\mathbb Z}}
\newcommand{\CC}{{\mathbb C}}
\newcommand{\Q}{{\mathbb Q}}
\newcommand{\IR}{{\mathbb R}}
\newcommand{\NN}{{\mathbb N}}
\newcommand{\FF}{{\mathbb F}}

\newcommand{\SU}{{\mathrm{SU}(2)}}
\newcommand{\SL}{{\mathrm{SL}_2(\CC)}}
\newcommand{\PSL}{{\mathrm{PSL}_2(\CC)}}

\newcommand{\trace}{{\rm tr}\,}

\newcommand{\rk}{\mathrm{rk}\,}

\newcommand{\I}{\mathbf{1}}

\newcommand{\bm}[1]{\mbox{\boldmath $#1$}} 

\newcommand{\smallbm}[1]{\mbox{\small \boldmath $#1$}} 

\newcommand{\im}{\mathop{\mathrm{im}}\nolimits}

\newcommand{\sll}{\mathfrak{sl}_2(\CC)}
\newcommand{\sllrho}{\mathfrak{sl}_2(\CC)_\rho}
\newcommand{\sllt}{\mathfrak{sl}_2(\CC(t))}
\newcommand{\sllrhot}{\mathfrak{sl}_2(\CC(t))_\rho}
\newcommand{\sllmt}{\mathfrak{sl}_2(\bm{t})}

\newcommand{\msll}[1]{\mathfrak{sl}_2(#1)}
\newcommand{\msllrho}[1]{\mathfrak{sl}_2(#1)_\rho}

\newcommand{\PolyTors}[2]{\Delta_{#1}^{#2}}

\newcommand{\RTors}[2]{\mathbb{T}^{#1}_{#2}}

\newcommand{\Tor}[3]{\mathrm{Tor} ({#1}, {#2}, {#3})}

\newcommand{\otAd}[2]{#1 \otimes Ad \circ #2}

\newcommand{\bbrack}[1]{\lbrack \! \lbrack #1 \rbrack \! \rbrack}

\newcommand{\cbasis}[2][c]{\mathbf{#1}^{#2}}

\newcommand{\bbasis}[2][b]{\mathbf{#1}^{#2}}
 
\newcommand{\hbasis}[2][h]{\mathbf{#1}^{#2}}

\newcommand{\hbasisrel}[1]{%
\hbasis{*}_{#1}, \hbasis{*}_{\bord #1}, \hbasis{*}_{(#1, \bord #1)}%
}

\newcommand{\lk}{\ell\mathit{k}}

\newcommand{\bord}{\partial}

\newcommand{\infcover}[1]{\overline{#1}}
\newcommand{\ucover}[1]{\widetilde{#1}}

\begin{document}

%
%


\title{
Twisted Alexander invariant and non--abelian Reidemeister torsion for hyperbolic three--dimensional manifolds with cusps
}


\author{J\'er\^ome Dubois \and Yoshikazu Yamaguchi}

\address{Institut de Math\'ematiques de Jussieu, 
Universit\'e Paris Diderot--Paris 7, 
UFR de Math\'ematiques, 
Case 7012, B\^atiment Chevaleret,
75205 Paris Cedex 13
France}
\email{dubois@math.jussieu.fr}

\address{Department of Mathematics, 
Tokyo Institute of Technology, 
2-12-1 Ookayama Meguro-ku, Tokyo 152-8551, Japan
}
\email{shouji@math.titech.ac.jp}

\date{\today}

\begin{abstract}
  We study a computational method of the hyperbolic Reidemeister torsion (also called in the literature the non--abelian Reidemeister torsion) considered by
  J. Porti for complete hyperbolic three--dimensional manifolds with cusps.  The
  derivative of the twisted Alexander invariant for a hyperbolic knot
  exterior gives the hyperbolic torsion.  We prove such a derivative
  formula of the twisted Alexander invariant for hyperbolic link
  exteriors like the Whitehead link exterior.  We provide the
  framework for the derivative formula to work, which consists of
  assumptions on the topology of the manifold and on the representations involved in the definition of the twisted Alexander invariant, and prove derivative formula in that context.  We also explore the symmetry properties (with sign) of
  the twisted Alexander invariant and prove that it is in fact a polynomial invariant, like the usual Alexander polynomial.
\end{abstract}

\maketitle



\section{Introduction}

Hyperbolic Reidemeister torsion, \emph{i.e.} the Reidemeister torsion
twisted by the adjoint representation associated to an irreducible
representation into $\SL$, was first introduced by D. Johnson in
unpublished notes~\cite{Johnson}. Later it was considered by
E. Witten~\cite{Witten:1991} as a symplectic volume form on the moduli
space of $\SU$-flat connections over two--dimensional surfaces and
developed to the case for closed three--manifolds by L.~Jeffrey and
 J.~Weitsman~\cite{JeffreyWeitsman:1993} and
J. Park~\cite{JPark:1997}. Furthermore, J. Porti~\cite{Porti:1997} applied the Reidemeister torsion
twisted by the adjoint representation to the study of three--dimensional hyperbolic manifolds eventually with cusps via $\SL$-character varieties.

Let $M$ be a hyperbolic three--dimensional manifold with
cusps. We consider the character variety $X(M)$ of $M$, which is in a
sense the ``algebraic quotient" of the representation space
$\mathrm{Hom}\big(\pi_1(M), \SL\big)$ under the action by conjugation. This set
has the structure of a complex algebraic affine set as it is proved in~\cite{CS:1983, LubotzkyMagid85}). The
\emph{geometric component} of $X(M)$ is the connected component which
contains the discrete and faithful representation of the complete
hyperbolic structure.

The construction of the hyperbolic torsion uses $\SL$-representations
whose conjugacy classes lie in the geometric component.
In general, it is not easy to compute the hyperbolic torsion directly from the definition because the twisted complex involved for the computation is not acyclic is that case.
But for hyperbolic knot exteriors, we can carry out the computation of the 
hyperbolic torsion by using the derivative of {\it the twisted Alexander invariant}
given by second author's work~\cite{YY:Fourier}. Such formula gives a link between the hyperbolic torsion, which is a ``non--acyclic" Reidemeister torsion, and another one which has the advantage to be computed using an acyclic complex. 
We call this procedure the {\it derivative formula} of the twisted Alexander invariant
and we use the terminology {\it non--abelian Reidemeister torsion} instead of the hyperbolic Reidemeister torsion
following author's previous works~\cite{JDFourier, DuboisVuYam09:RTorsionTwistKnots}.

This paper provides the generalized framework to compute the non--abelian Reidemeister torsion of a hyperbolic three--dimensional manifold with cusps 
by using the derivative of the twisted Alexander invariant with multivariables.
We proceed from the technical conditions required by our framework to the computation procedure of 
the non--abelian Reidemeister torsion.
Our framework requires the following three kinds of assumptions:
\begin{itemize}
\item topological conditions on the hyperbolic three--dimensional manifold $M$ whose boundary consists in the disjoint union of $b$ two--dimensional tori: $$\partial M = \bigcup_{\ell =1}^b T_\ell^2;$$
\item conditions on the surjective homomorphism $$\varphi\colon \pi_1(M) \to \ZZ^n = \langle t_1, \ldots, t_n \; |\; t_it_j = t_jt_i \;(\forall\, i,j)\,\rangle$$ which gives the variables of the twisted Alexander invariant and;
\item conditions on an irreducible $\SL$-character $\rho\colon \pi_1(M) \to \SL$ which lies in the geometric component of the character variety.
\end{itemize}
These assumptions are referred to as the symbols $(A_M)$,
$(A_\varphi)$ and $(A_\rho)$.  Under these assumptions, we show the following four results:
\begin{itemize}
\item the non--abelian Reidemeister torsion $\RTors{M}{\smallbm{\lambda}}(\rho)$ and the twisted Alexander
  invariant $\PolyTors{M}{\otAd{\varphi}{\rho}}(t_1, \ldots, t_n)$ are well--defined;
\item the derivative of the twisted Alexander invariant with
  multivariables gives the non--abelian Reidemeister torsion. More precisely, we have:
  \[
  \lim_{t_1, \ldots, t_n \longrightarrow 1}\; 
    \frac{
      \PolyTors{M}{\otAd{\varphi}{\rho}}(t_1, \ldots, t_n)
    }{
      \prod_{\ell = 1}^b \big( t_1^{a_1^{(\ell)}} \cdots t_n^{a_n^{(\ell)}} - 1 \big )
    } 
    = (-1)^b \cdot \RTors{M}{\smallbm{\lambda}}(\rho).
    \]
    Here $b$ denotes the number of tori components of the boundary $\partial M$ of $M$ and $\varphi\big(\pi_1(T_\ell^2)\big) = \big\langle t_1^{a_1^{(\ell)}} \cdots
  t_n^{a_n^{(\ell)}}\big \rangle$ for some positive integers $a_1^{(\ell)}, \ldots, a_n^{(\ell)}$;
  \item under some technical conditions on the representations $\varphi$ and $\rho$, we prove that the twisted Alexander invariant $\PolyTors{M}{\otAd{\varphi}{\rho}}(t_1, \ldots, t_n)$, which a priori lies in the fraction field $\CC(t_1, \ldots, t_n)$, is in fact  contained in $\CC[t_1^{\pm 1}, \ldots, t_n^{\pm 1}]$ (up to a factor $t_1^{m_1} \cdots t_n^{m_n}$ for some integers $m_1, \ldots, m_n$), and;
\item we give the computation example of the non--abelian Reidemeister
  torsion for the Whitehead link exterior by using this generalized
  formula.
\end{itemize}

The advantage of our generalization of the derivative formula lies in
the fact that we can obtain a sequences of hyperbolic knot exteriors
by Dehn filling a hyperbolic link exterior with some solid tori.
Since twist knots are obtained by Dehn filling the Whitehead link
exterior, the computations result of the non--abelian Reidemeister
torsion for the Whitehead link exterior is useful to observe the
asymptotic behavior of the non--abelian Reidemeister torsions for the
twist knot exteriors, as observed in~\cite{DuboisVuYam09:RTorsionTwistKnots} by
V. Huynh and the authors.

It is expected that the twisted Alexander invariant in our framework
has other properties than the derivative formula.  We also prove a duality formula (with sign) for the twisted Alexander invariant which can be compared with the well--known symmetry property of the usual Alexander polynomial. More precisely, if $M$ is a link exterior $M = E_L = S^3 \setminus N(L)$, then
the polynomial torsion of $E_L$ satisfies the following formula:
\[
    \PolyTors{E_L}{\otAd{\varphi}{\rho}}(t^{-1}) 
      =(-1)^b \PolyTors{E_L}{\otAd{\varphi}{\rho}}(t). 
\]

\section*{Organization}
The outline of the paper is as follows.
Section~\ref{section:preliminaries} deals with some reviews on the
sign--determined Reidemeister torsion for a manifold and on the
multiplicative property of Reidemeister torsions (the Multiplicativity
Lemma) which is the main tool for computing Reidemeister torsions by
using a cut and past argument. In
Section~\ref{section:def_polynomial_torsion}, we set down technical
assumptions used in the whole of this paper, prove that a certain
twisted homology vanishes and then give the definition of the twisted
Alexander invariant for a hyperbolic manifold with boundary.
Section~\ref{section:exemples} presents some examples of computations
of the twisted Alexander invariant for the figure eight knot exterior
and for the Whitehead link exterior.  Section~\ref{section:naturality}
gives a reduction of the number of variables in the twisted Alexander
invariant when we change coefficients in the twisted chain complex in
the definition of the torsion.  Section~\ref{section:derform} is
devoted to the derivative formula which gives a bridge joining the
twisted Alexander invariant and the non--abelian Reidemeister torsion in the adjoint
representation associated with an $\SL$-representation in the geometric
component.  In Section~\ref{section:poly}, we prove that the twisted
Alexander invariant is generically a polynomial by using a cut and
paste argument and the multiplicative property of Reidemeister
torsions (see Theorem~\ref{theorem:polyvstorsion}). The symmetry
properties (with sign) of this polynomial torsion are investigated in
Section~\ref{section:symmetries} (see
Theorem~\ref{theorem:symmetry_torsion}).

\section*{Acknowledgments}

The authors gratefully acknowledge the many helpful ideas and
suggestions of Joan Porti during the preparation of the paper.  They
also wish to express their thanks to C.~Blanchet and V.~Maillot for
several helpful comments.  
Our thanks also go to T.~Morifuji, K.~Murasugi and S.~Friedl
for giving us helpful informations 
and comments regarding this work.
The paper was conceived during J.D. visited
the CRM. He thanks the CRM for hospitality.  The first author (J.D.) is partially supported by the French ANR project ANR-08-JCJC-0114-01.The second author (Y.Y.)
is partially supported by the GCOE program at Graduate School of
Mathematical Sciences, University of Tokyo. Y.Y. visited CRM and IMJ
while writing the paper. He thanks CRM and IMJ for their hospitality.

\section{Preliminaries}
\label{section:preliminaries}
\subsection{The Reidemeister torsion}
We review the basic notions and results about the sign--determined
Reidemeister torsion introduced by V. Turaev which are needed in this
paper. Details can be found in Milnor's survey~\cite{Milnor:1966} and
in Turaev's monograph~\cite{Turaev:2002}.

\subsubsection*{Torsion of a chain complex}
Let 
$C_* = (
  0 \to 
    C_n \xrightarrow{d_n} 
    C_{n-1} \xrightarrow{d_{n-1}} 
    \cdots \xrightarrow{d_1} 
    C_0 \to 
  0 
)$ 
be a chain complex of finite dimensional vector spaces over a field
$\FF$. Choose a basis $\cbasis{(i)}$ of $C_i$ and a basis $\hbasis{i}$
of the $i$-th homology group $H_i(C_*)$. The torsion of $C_*$ with
respect to these choices of bases is defined as follows.

For each $i$, let $\bbasis{i}$ be a set of vectors in $C_{i}$ such that
$d_{i}(\bbasis{i})$ is a basis of 
$B_{i-1}= \im(d_{i} \colon C_{i} \to C_{i-1})$ and 
let $\hbasis[\tilde{h}]{i}$ denote a lift of
$\hbasis{i}$ in $Z_i = \ker(d_{i} \colon C_i \to C_{i-1})$. The set of
vectors $d_{i+1}(\bbasis{i+1})\hbasis[\tilde{h}]{i}\bbasis{i}$ is a
basis of $C_i$. Let
$[d_{i+1}(\bbasis{i+1})\hbasis[\tilde{h}]{i}\bbasis{i}/\cbasis{i}] \in \FF^*$ 
denote the determinant of the transition matrix between those
bases (the entries of this matrix are coordinates of vectors in
$d_{i+1}(\bbasis{i+1})\hbasis[\tilde{h}]{i}\bbasis{i}$ with respect to
$\cbasis{i}$). The \emph{sign-determined Reidemeister torsion} of
$C_*$ (with respect to the bases $\cbasis{*}$ and $\hbasis{*}$) is the
following alternating product (see~\cite[Definition 3.1]{Turaev:2000}):
\begin{equation}\label{def:RTorsion}
  \Tor{C_*}{\cbasis{*}}{\hbasis{*}} 
  = (-1)^{|C_*|} \cdot  
  \prod_{i=0}^n 
  [d_{i+1}(\bbasis{i+1})\hbasis[\tilde{h}]{i}\bbasis{i}/\cbasis{i}]^{(-1)^{i+1}} 
  \in \FF^*.
\end{equation}
Here $$|C_*| = \sum_{k\mygeq 0} \alpha_k(C_*) \beta_k(C_*),$$ where
$\alpha_k(C_*) = \sum_{i=0}^k \dim C_i$ and 
$\beta_k(C_*) = \sum_{i=0}^k \dim H_i(C_*)$.

The torsion $\Tor{C_*}{\cbasis{*}}{\hbasis{*}}$ does not depend on the
choices of $\bbasis{i}$ nor on the lifts $\hbasis[\tilde{h}]{i}$.  Note that if
$C_*$ is acyclic (\ie~if $H_i = 0$ for all $i$), then $|C_*| = 0$.

\subsubsection*{Torsion of a CW-complex}
Let $W$ be a finite CW-complex and $(V, \rho)$ be a pair of a vector
space with an inner product over $\FF$ and a homomorphism of $\pi_1(W)$
into $Aut(V)$.  The vector space $V$ turns into a right $\ZZ[\pi_1(W)]$-module denoted $V_{\rho}$ by using the right
action of $\pi_1(W)$ on $V$ given by $v \cdot \gamma = \rho(\gamma)^{-1}(v)$, for
$v \in V$ and $\gamma \in \pi_1(W)$.  The complex of the universal cover 
with integer coefficients $C_*(\ucover{W}; \ZZ)$ also inherits a left
$\ZZ[\pi_1(W)]$-module via the action of $\pi_1(W)$ on
$\ucover{W}$ as the covering group.
We define the $V_\rho$-twisted
chain complex of $W$ to be
\[
C_*(W; V_\rho) = V_{\rho} \otimes_{\ZZ[\pi_1(W)]} C_*(\ucover{W}; \ZZ).
\]
The complex $C_*(W; V_\rho)$ computes the {\it $V_\rho$-twisted homology} of $W$
which we denote as $H_*(W;V_\rho)$.

Let $\left\{e^{i}_1, \ldots, e^{i}_{n_i}\right\}$ be the set of
$i$-dimensional cells of $W$. We lift them to the universal cover and
we choose an arbitrary order and an arbitrary orientation for the
cells $\left\{ {\tilde{e}^{i}_1, \ldots, \tilde{e}^{i}_{n_i}}
\right\}$. If we let $\{\bm{v}_1, \ldots, \bm{v}_m\}$ be an
orthonormal basis of $V$, then we consider the corresponding basis
$$
\cbasis{i} = 
  \left\{ 
    \bm{v}_1 \otimes \tilde{e}^{i}_{1} , 
    \ldots, 
    \bm{v}_m \otimes \tilde{e}^{i}_{1}, 
    \cdots, 
    \bm{v}_1 \otimes \tilde{e}^{i}_{n_i}, 
    \ldots, 
    \bm{v}_m \otimes \tilde{e}^{i}_{n_i}
  \right\}
$$ 
of 
$C_i(W; V_\rho) = V_\rho \otimes_{\ZZ[\pi_1(W)]} C_*(\ucover{W};\ZZ)$. 
We call the basis $\cbasis{*} = \oplus_{i} \cbasis{i}$ 
{\it a geometric basis} of $C_*(W;V_\rho)$.
Now choosing for each $i$ a basis $\hbasis{i}$ of the
$V_\rho$-twisted homology $H_i(W; V_\rho)$, we can compute the torsion
$$\Tor{C_*(W; V_\rho)}{\cbasis{*}}{\hbasis{*}} \in \FF^*.$$

The cells 
$\left\{ \left. 
  \tilde{e}^{i}_j \, 
         \right|\, 
  0 \myleq i \myleq \dim W, 
  1 \myleq j \myleq n_i
\right\}$ are in one--to--one correspondence with the cells of $W$, 
their order and orientation induce an order and 
an orientation for the cells 
$\left\{ \left. 
  \tilde{e}^{i}_j \,
         \right|\, 
  0 \myleq i \myleq \dim W, 
  1 \myleq j \myleq n_i
\right\}$. 
Again, corresponding to these
choices, we get a basis $\cbasis{i}_\IR$ over $\IR$ of $C_i(W; \IR)$.

Choose an \emph{homology orientation} of $W$, which is an orientation
of the real vector space $H_*(W; \IR) = \bigoplus_{i\mygeq 0}
H_i(W; \IR)$. Let $\mf{o}$ denote this chosen
orientation. Provide each vector space $H_i(W; \IR)$ with a reference
basis $\hbasis{i}_\IR$ such that the basis 
$\left\{ {\hbasis{0}_\IR, \ldots, \hbasis{\dim W}_\IR} \right\}$ of 
$H_*(W; \IR)$ is {positively oriented} with respect to $\mf{o}$. 
Compute the sign--determined Reidemeister torsion 
$\Tor{C_*(W; \IR)}{\cbasis{*}_\IR}{\hbasis{*}_\IR} \in \IR^*$ of 
the resulting based and homology based chain complex and consider its sign 
$$\tau_0 =
  \mathrm{sgn}\left(
    \Tor{C_*(W; \IR)}{\cbasis{*}_\IR}{\hbasis{*}_\IR}
   \right)
\in \{\pm 1\}.$$

We define the \emph{sign--refined twisted Reidemeister torsion} of $W$
{(with respect to $\hbasis{*}$ and $\mf{o}$)} to be
\begin{equation}\label{eqn:TorsionRaff}
  \tau_0 \cdot 
  \Tor{C_*(W; V_\rho)}{\cbasis{*}}{\hbasis{*}} \in \FF^*.
\end{equation}
This definition only depends on the combinatorial class of $W$, the
conjugacy class of $\rho$, the choice of $\hbasis{*}$ and the
{homology} orientation $\mf{o}$. It is independent of the
orthonormal basis of $V$, of the choice of the lifts
$\tilde{e}^{i}_j$, and of the choice of the positively oriented basis
of $H_*(W; \IR)$. Moreover, it is independent of the order and
orientation of the cells (because they appear twice).

\begin{remark}
  In particular, if the Euler characteristic $\chi(W)$ is zero, then
  we can use any basis of $V$. If we change the basis of $V$ by another one, then the torsion is multiplicated 
by  the determinant of the bases change matrix to the power $\chi(W)$.
\end{remark}

One can prove that the sign--refined Reidemeister torsion is invariant 
under cellular subdivision, 
homeomorphisms and simple homotopy equivalences. In fact,
it is precisely the sign $(-1)^{|C_*|}$ in Equation~(\ref{def:RTorsion})
which ensures all these important invariance properties to hold (see~\cite{Turaev:2002}).

\subsection{The Multiplicativity Lemma for torsions}
\label{subsection:multiplicativitylemma}

In this section, we briefly review the Multiplicativity Lemma for
Reidemeister torsions (with sign). 

First, we review the notion of \emph{compatible bases}. Let $0 \to E'
\xrightarrow{i} E \xrightarrow{j} E'' \to 0$ be a short exact sequence
of finite dimensional vector spaces and let $s$ denotes a section of
$j$. Thus, $i \oplus s : E' \oplus E'' \to E$ is an isomorphism. We
equip the three vector spaces $E'$, $E$ and $E''$ respectively with
the following three bases : $\mathbf{b}' = \left( b'_1, \ldots,
  b'_p\right)$, $\mathbf{b} = \left(b_1, \ldots, b_n\right)$, and
$\mathbf{b}'' = \left(b''_1, \ldots, b''_q\right)$. With such
notation, one has $n = p+q$, and we say that the bases $\mathbf{b}'$,
$\mathbf{b}$ and $\mathbf{b}''$ are compatible if the isomorphism $i
\oplus s : E' \oplus E'' \to E$ has determinant $1$ in the bases
$\mathbf{b}' \cup \mathbf{b}'' = \left( b'_1, \ldots, b'_p, b''_1,
  \ldots, b''_q\right)$ of $E' \oplus E''$ and $\mathbf{b}$ of $E$. If
it is the case, we write $\mathbf{b} \sim \mathbf{b}' \cup
\mathbf{b}''$.

Let us now review the multiplicativity property of the Reidemeister torsion (with sign).

\begin{MultLemma}[Lemma 3.4.2 in~\cite{Turaev86:torsion_in_knot_theory}]
  \label{lemma:Multiplicativite}
  Let
  \begin{equation}\label{eqn:short_exact_seq_complexes}
      0 \to C'_* \to C_* \to C_*'' \to 0
  \end{equation}
  be an exact sequence of chain complexes. Assume that $C'_*$, $C_*$
  and $C''_*$ are based and homology based. For all $i$, let
  $\cbasis[c']{i}$, $\cbasis[c]{i}$ and $\cbasis[c'']{i}$ denote the
  reference bases of $C_i'$, $C_i$ and $C_i''$ respectively.
  Associated to~(\ref{eqn:short_exact_seq_complexes}) is the long
  sequence in homology
  \[
    \cdots \to 
    H_i(C'_*) \to 
    H_i(C_*)  \to 
    H_i(C''_*) \to 
    H_{i-1}(C'_*) \to 
    \cdots 
  \]
  Let $\mathcal{H}_*$ denote this acyclic chain complex and base
  $\mathcal{H}_{3i+2} = H_i(C'_*)$, $\mathcal{H}_{3i+1} = H_i(C_*)$
  and $\mathcal{H}_{3i} = H_i(C''_*)$ with the reference bases of
  $H_i(C'_*)$, $H_i(C_*)$ and $H_i(C''_*)$ respectively.  If for all
  $i$, the bases $\cbasis[c']{i}$, $\cbasis{i}$ and $\cbasis[c'']{i}$
  are compatible, 
  \ie~$\cbasis{i} \sim \cbasis[c']{i} \cup \cbasis[c'']{i}$, then
  \begin{align*}
    \lefteqn{\Tor{C_*}{\cbasis{*}}{\hbasis{*}}} &\\
    &= (-1)^{\alpha(C'_*, C''_*) + \varepsilon(C'_*, C_*, C''_*)} \;
    \Tor{C'_*}{\cbasis[c']{*}}{\hbasis[h']{*}}
    \Tor{C''_*}{\cbasis[c'']{*}}{\hbasis[h'']{*}}
    \Tor{\mathcal{H}_*}{\{\hbasis[h']{*}, \hbasis{*},\hbasis[h'']{*}\}}{\emptyset}
  \end{align*}
  where 
  \begin{align*}
    \alpha(C'_*, C''_*) &= 
    \sum_{i \mygeq 0} \alpha_{i-1}(C'_*) \alpha_i(C''_*) \in \ZZ/2\ZZ 
  \quad \hbox{and}\\
    \varepsilon(C'_*, C_*, C''_*) &= 
    \sum_{i \mygeq 0}
    \{(\beta_i(C_*)+1)(\beta_i(C_*') + \beta_i(C_*'')) + 
      \beta_{i-1}(C_*')\beta_i(C_*'')\}
    \in \ZZ/2\ZZ.
  \end{align*}
\end{MultLemma}
The proof is a careful computation based on linear algebra,
see~\cite[Lemma 3.4.2]{Turaev86:torsion_in_knot_theory}
and~\cite[Theorem 3.2]{Milnor:1966}.  This lemma appears to be a very
powerful tool for computing Reidemeister torsions. It will be used all
over this paper.

\section{Definition of the polynomial torsion}
\label{section:def_polynomial_torsion}

In this section, we define the twisted Alexander invariant and called it for short the polynomial torsion. This invariant is the twisted Alexander invariant with coefficients in the adjoint representation associated to a character which lies in the geometric component of the character variety of the three--manifold. We define it following the presentation given by Friedl and Vidussi in their survey~\cite{FriedlVidussi:survey} using Reidemeister torsions theory. 

Hereafter $M$ denotes a compact and connected hyperbolic three--dimensional
manifold such that its boundary $\partial M$ consists in a disjoint
union of $b$ two--dimensional tori: $$\partial M = T^2_1 \cup \ldots \cup T^2_b.$$

In the sequel,  $\rho$ denotes a representation of $\pi_1(M)$ into $\SL$.
The composition of $\rho$ with the adjoint action $Ad$ of $\SL$ on $\sll$ 
gives us the following representation:
\begin{align*}
  Ad \circ \rho \colon \pi_1(M) &\to Aut(\sll) \\
 \gamma &\mapsto (v \mapsto \rho(\gamma) v \rho(\gamma)^{-1})
\end{align*}
We let $\sllrho$ denote the right $\ZZ[\pi_1(M)]$--module $\sll$ 
via the action $Ad \circ \rho^{-1}$.

Now we introduce the two different twisted chain complexes which will be considered throughout this paper. 

 The first twisted complex under consideration is the complex $C_*(M; \sllrho)$ defined by:
 \begin{equation}
  C_*(M; \sllrho) = \sllrho \otimes_{\ZZ[\pi_1(M)]} C_*(\ucover{M}; \ZZ).
\end{equation}
This chain complex is called {\it the $\sllrho$-twisted
  chain complex of $M$}. The twisted chain complex $C_*(M; \sllrho)$ computes the so--called {\it 
  $\sllrho$-twisted homology} 
 denoted by 
$H_*^\rho(M) = H_*(M; \sll_\rho)$.  It is well--known that for a three--manifold $M$ with non--empty boundary one has $\dim H_1^\rho(M) \geqslant b$. Thus $C_*(M; \sllrho)$ is never acyclic for three--manifolds with non--empty boundary. We will use the symbol $\RTors{M}{}$
to denote the sign--refined Reidemeister torsion of $C_*(M;\sllrho)$.

Next we introduce a twisted chain complex with some variables.
It will be done by using a $\ZZ[\pi_1(M)]$--module with variables 
to define a new twisted chain complex.
We regard $\ZZ^n$ as the multiplicative group generated by $n$ variables
$t_1, \ldots, t_n$, \ie, 
$$\ZZ^n = \left\langle t_1, \ldots, t_n \,|\, t_it_j = t_jt_i \, (\forall i, j)\right\rangle$$
and 
consider  a surjective homomorphism $\varphi\colon \pi_1(W)\to\ZZ^n$.
We often abbreviate the $n$ variables $(t_1, \ldots, t_n)$ to $\bm{t}$ and 
the rational functions $\CC(t_1, \ldots, t_n)$ to $\CC(\bm{t})$.
Moreover, we write 
$\msll{t_1, \ldots, t_n} = \sllmt$ for $\CC(\bm{t}) \otimes_{\CC} \sll$ 
for brevity.  
Note that $\sllmt$ is naturally identified with $\msll{\CC(\bm{t})}$ 
which is the vector space of trace free matrices whose components are rational functions 
in $\CC(\bm{t}) = \CC(t_1, \ldots, t_n)$.
The group $\pi_1(M)$ acts on $\sllmt$ via the following action: 
$$
\otAd{\varphi}{\rho} : 
  \pi_1(M) \to Aut(\CC(\bm{t}) \otimes \sll) = Aut(\sllmt).
$$ 
Thus, $\sllmt$ inherits the structure of a right $\ZZ[\pi_1(M)]$-module, $\sllmt_\rho$, and we consider the associated twisted chain $C_*(M; \sllmt_\rho)$ given by:
$$
C_*(M; \sllmt_\rho) = \sllmt \otimes_{\ZZ[\pi_1(M)]} C_*(\ucover{M}; \ZZ)
$$
where $ f \otimes v \otimes \gamma \cdot \sigma$ is identified with 
$f \varphi(\gamma) \otimes Ad_{\rho(\gamma)^{-1}}(v) \otimes \sigma$ for
any 
$\gamma \in \pi_1(M)$, 
$\sigma \in C_*(\ucover{M};\ZZ)$, 
$v \in \sll$ and 
$f \in \CC(\bm{t})$.
We call this complex {\it the $\sllmt_\rho$-twisted chain complex} of $M$, and its homology is denoted $H_*(M; \sllmt_\rho)$.

Now we define geometric bases. We choose a basis of $\sll$, for example, 
\begin{equation}\label{eqn:basis_sll}
\{E, H, F\} 
=
\left\{ 
\left(
  \begin{array}{cc}
    0&1\\
    0&0
  \end{array}
\right),
\left(
  \begin{array}{cc}
    1&0\\
    0&-1
  \end{array}
\right), 
\left(
  \begin{array}{cc}
    0&0\\
    1&0
  \end{array}
\right)
\right\},
\end{equation}
a geometric basis $\cbasis{*}$ of $C_*(M; \sllrho)$ is obtained from 
the CW--structure of $M$. The geometric basis $\cbasis{*}$ automatically gives
us the geometric basis $1 \otimes \cbasis{*}$ of $C_*(M; \sllmt_\rho)$. In
all this paper these two bases will be abusively denoted with the same
notation.

\begin{definition}\label{def:AlexTorsion}
Fix a homology orientation on $M$. 
  If $C_*(M;\sllmt_\rho)$ is acyclic, then the sign--refined Reidemeister torsion
  of $C_*(M;\sllmt_\rho)$:
  $$
  \PolyTors{M}{\otAd{\varphi}{\rho}}(t_1, \ldots, t_n) = 
    \tau_0 \cdot
    \Tor{C_*(M;\sllmt_\rho)}{\cbasis{*}}{\emptyset} 
    \in \CC(t_1, \ldots, t_n) \setminus \{0\}.
  $$
  is called the \emph{twisted Alexander invariant} (or the \emph{polynomial torsion} for short) of $M$.
\end{definition}

Note that the sign--refined Reidemeister torsion
$\PolyTors{M}{\otAd{\varphi}{\rho}}$ is determined up to a factor
$t_1^{m_1} \cdots t_n^{m_n}$ such as the classical Alexander polynomial.

\begin{example}
  Suppose that $M$ is the knot exterior $E_K = S^3 \setminus N(K)$ of
  a knot $K$ in $S^3$ where $N(K)$ is an open tubular neighbourhood of
  $K$. If the representation $\rho \in \mathrm{Hom}(\pi_1(E_K); \Q)$
  is the trivial homomorphism and $\varphi$ is the abelianization of
  $\pi_1(E_K)$, \ie, $\varphi:\pi_1(E_K) \to H_1(E_K;\ZZ) \simeq
  \langle t\rangle$, then the twisted chain complex $C_*(E_K;
  \Q(t)_\rho)$ is acyclic and the torsion $\Delta_{E_K}^{\varphi
    \otimes Ad \circ \rho}(t)$ is the Alexander polynomial divided by
  $(t-1)$ (see~\cite{Milnor:1962} and~\cite{Turaev:2002}).
\end{example}

\subsection{Technical assumptions}
\label{section:hyp}

In this subsection, we give some sufficient conditions on the compact
hyperbolic three--manifold $M$ whose boundary consists of
a disjoint union of tori and on the representations $\varphi$ and
$\rho$ which assure the acyclicity of the twisted chain complex
$C_*(M;\sllmt_\rho)$. 

We require assumptions on the topology of $M$,
on the surjective homomorphism $\varphi$ of $\pi_1(M)$ onto $\ZZ^n$ and
on the $\SL$-representation $\rho$. The assumption on $\varphi$ is also
related to the topology of $M$, we will show this in the second half
of this subsection.

We also prove that if $M$ is a hyperbolic knot exterior,
$\varphi\colon \pi_1(M) \to \ZZ$ the abelianization and $\rho\colon
\pi_1(M) \to \SL$ the holonomy (\ie~the discrete and faithful
representation), then all our assumptions are satisfied.

\subsubsection{Topological assumption for $M$}

Let $b$ be the number of components of $\bord
M$.  We let $T^2_\ell$ denotes the $\ell$--component of $\bord M$.  The usual inclusion $i:\bord M \to M$
  induces an homomorphism $i_* \colon H_1(\bord M;\ZZ) \to
  H_1(M;\ZZ)$. First we assume
the following condition on the homology group $H_1(M;\ZZ)$:
\begin{itemize}
\item[$(A_M)$] the homomorphism  $i_* \colon H_1(\bord M;\ZZ) \to
  H_1(M;\ZZ)$ is onto and its restriction  $(i|_{T^2_\ell})_*$ to the $\ell$-th component of $\partial M$ 
  has rank one for all $\ell$.
\end{itemize}
Thus we can choose two closed loops $\mu_\ell$ and $\lambda_\ell$ on $T^2_\ell$ 
such that the homology classes $\bbrack{\mu_\ell}$ and $\bbrack{\lambda_\ell}$ form a basis of 
$H_1(T^2_\ell;\ZZ)$ and the image $i_*(\bbrack{\mu_\ell})$ generates the subgroup $\im (i|_{T^2_\ell})_*$
and $\bbrack{\lambda_\ell}$ generates the kernel of $(i|_{T^2_\ell})_*$. We call $\mu_\ell$ a meridian and $\lambda_\ell$ a longitude.

We let $\bm{\lambda}$ denote the set 
$(\lambda_1, \ldots, \lambda_b) \subset \partial M$ of 
such generators of $\ker \varphi_{|\pi_1(T^2_\ell)}$
and call $\bm{\lambda}$ the multi--longitude curve.

\begin{remark}
  From the homology long exact sequence of $(M, \bord M)$, it follows
  that $b_1(M) = b$ and $H_1(M;\ZZ)$ has no--torsion.
\end{remark}

\begin{example}[Knot exteriors]\label{ex:knot}
Suppose that $M$ is the exterior $E_K$ of a
hyperbolic knot $K$ in $S^3$. Here $E_K = S^3 \setminus N(K)$ where $N(K)$ is an
open tubular neighbourhood of $K$.  Then $M = E_K$ satisfies the
condition~$(A_M)$. This is due to the existence of a Seifert
surface of the knot.
\end{example}

\begin{remark}[Link exteriors]

 Let $L$ be a hyperbolic link such that each
components of the link bounds a Seifert surface missing the other
components.  This condition for a link $L = K_1 \cup \ldots \cup K_b$
is equivalent to that the linking numbers $\lk(K_i, K_j)$ are zero for
all $i, j$. It is the reason that we can obtain the required Seifert
surface by first choosing an arbitrary Seifert surface for $K_{i_k}$
and then getting rid of the intersections by adding tubes as in
Fig~\ref{fig:whitehead_link}.  The intersections of such Seifert surfaces
and the boundary $\bord E_L$ form a set of longitudes.

Note that it is not necessarily the case that Seifert surfaces are
disjoint if each component of a link $L = K_1 \cup \ldots \cup K_b$
bounds a Seifert surface missing the other components.  Links whose
components do bound disjoint Seifert surfaces are called 
{\it boundary links}. 
For example, the Whitehead link $L = K_1 \cup K_2$ has the
linking number $\lk(K_1, K_2) = 0$.  Hence there exist two Seifert
surfaces $F_i$ ($i=1, 2$) such that $\bord F_i = K_i$ and $F_i \cap
K_j = \emptyset$\, ($i \not = j$).  But in fact the Whitehead link
does not bound disjoint Seifert surfaces since the Whitehead link is
not a boundary link, for more details see~\cite[Chapter 5 E]{Rolfsen}.

In this paper, three--manifold under considerations has the same properties about the homology group like as 
those of a boundary link exterior.
\end{remark}

\begin{figure}[!h]
  \begin{center}
    \includegraphics[scale=.5]{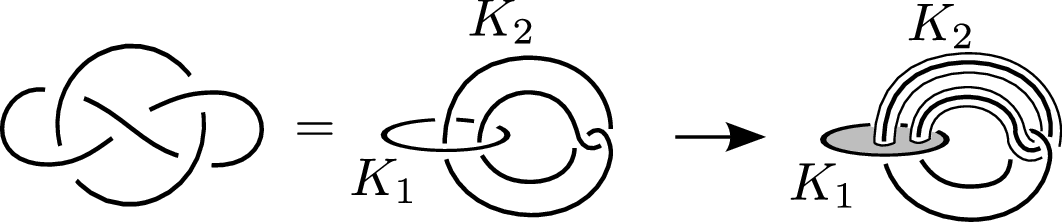}
  \end{center}
  \caption{The Whitehead link}
  \label{fig:whitehead_link}
\end{figure}

\subsubsection{Assumption on $\varphi$}

We suppose that 
$$\varphi \colon \pi_1(M) \to \ZZ^n = 
\left\langle t_1, \ldots, t_n \;|\; t_i t_j = t_j t_i, \; \forall i, j \; \right\rangle$$
 is a
surjective homomorphism which satisfies the following assumption (in
the case of manifolds with boundary)
\begin{itemize}
\item[$(A_\varphi)$] the restriction $\varphi_{|\pi_1(T^2_\ell)}$ of
  $\varphi$ to the $\ell$--th boundary component of $\partial M$ is non--trivial and
  has rank 1, \ie
  $$\varphi(\pi_1(T^2_\ell)) = 
  \langle t_1^{a_1^{(\ell)}} \cdots t_n^{a_n^{(\ell)}} \rangle$$ 
  for some $(a_1^{(\ell)}, \cdots, a_n^{(\ell)}) \in \NN^n\setminus \{0\}$.
\end{itemize} 

\begin{remark}
  Every homomorphism from $\pi_1(M)$ to an abelian group factors
  through the abelianization $H_1(M;\ZZ)$ of $\pi_1(M)$, \ie, the
  following diagram is commutative:
  $$
  \xymatrix@R=20pt{
    \pi_1(M) \ar[rd]_{\varphi} \ar[r] & 
    H_1(M;\ZZ) \ar[d]^{\varphi_h}\\
     & 
    \ZZ^n 
  }
  $$
  When we consider a surjective homomorphism $\varphi$ onto $\ZZ^n$,
  the induced homomorphism $\varphi_h$ is also surjective. Together
  with the condition $(A_M)$, this means that \emph{$n$ must be not greater than  $b_1(M) = b$}.  Moreover we have the null--homologous closed
  curve $\lambda_\ell$ on each $T^2_\ell$, every homomorphism $\varphi$ has
  the multi--longitude consisting of these $\lambda_\ell$.
\end{remark}

\begin{example}[Abelianization representation]\label{ex:ab}
 Suppose that $M$ is the exterior $E_K$ of a knot $K \subset S^3$ or a link exterior $E_L$ of a link $L \subset S^3$ whose linking number between arbitrary two components is zero. The
abelianization $\varphi\colon \pi_1(M) \to H_1(M; \ZZ) \simeq \ZZ^b$ satisfies Assumption $(A_\varphi)$.  From Assumption~$(A_M)$, it follows that $H_1(M;\ZZ)
= \langle \mu_1, \ldots, \mu_b\rangle$.  Assumption ~$(A_\varphi)$
means that each entry of any representative matrix of $\varphi_h$ is
non-zero.
\end{example}

\subsubsection{Assumption on $\SL$-representations}

It is required for an $\SL$-representation $\rho$ of $\pi_1(M)$ to be
a {\it ``generic''} representation, which essentially means that
$\rho$ lies in {\it the geometric components} of {\it the character
  variety}.
  
The character variety of $\pi_1(M)$ is the set of characters of
$\SL$-representations. Here the character of an $\SL$-representation
$\rho$ is a map $\pi_1(M) \to \CC$ given by the assignment $\gamma
\mapsto \trace \rho(\gamma)$ for all $\gamma \in \pi_1(M)$, where $\trace$ denotes the usual trace of square matrices.  This set
has a structure of an affine algebraic variety (refer to
\cite{CS:1983}) denoted by $X(M)$. For a complete hyperbolic manifold,
the character variety $X(M)$ contains the distinguished components
related to the complete hyperbolic structure.  These components are
defined by containing a lift of the holonomy representation $\pi_1(M)
\to \PSL$ (\ie, the discrete and faithful representation)  determined by the complete hyperbolic structure.
We call these components the geometric components.

Here $\bm{\lambda}$ is the multi--longitude corresponding to $\varphi$.
The precise definition is given as follows.

\begin{notation*}
  For each boundary component $T^2_\ell$ of $M$, we fix a generator
  $P_\ell^\rho$ of the homology group $H_0(T^2_\ell; \sllrho)$, \ie~a
  non--trivial vector $P_\ell^\rho \in \sll$ which satisfies
  $Ad_{\rho(g)}(P_\ell^\rho) = P_\ell^\rho$ for all $g \in \pi_1(T^2_\ell)$
  (for more details, see~\cite[Section 3.3.1]{Porti:1997}).
\end{notation*}

\begin{definition}
  For the multi--longitude $\bm{\lambda} = \{\lambda_1, \ldots,
  \lambda_b\}$, an irreducible representation $\rho$ is called
  $\bm{\lambda}$-regular if $\rho$ satisfies the following conditions:
  \begin{enumerate}
  \item for each boundary component $T^2_\ell$ of $\bord M$, the restriction of
    $\rho$ to $\pi_1(T^2_\ell)$ is non--trivial;
  \item the following homomorphism, induced from all inclusions $\lambda_\ell \hookrightarrow M$ $(1 \myleq \ell \myleq b)$,
    $$\mathop{\bigoplus}_{\ell = 1}^b H_1(\lambda_{\ell};\sllrho) \to H_1(M;\sllrho)$$
    is surjective and;
  \item if $\trace \rho(\pi_1(T^2_\ell)) \subset \{\pm 2\}$, 
     then
    $\rho(\lambda_{\ell}) \not = \pm \I$.
  \end{enumerate}
\end{definition}

\begin{remark}
  The chain $P^\rho_\ell \otimes \lambda_{\ell}$ becomes a cycle \ie,
  it defines a homology class in $H_1(\lambda;\sllrho)$.
  By~\cite[Proposition 3.22 and Corollaire 3.21]{Porti:1997}, for a
  $\bm{\lambda}$-regular representation $\rho$, we have the following
  bases of the twisted homology groups:
  \begin{itemize}
  \item the homology group $H_1(M;\sllrho)$ has a basis $\left\{
    \bbrack{P^\rho_1 \otimes \lambda_1}, \ldots, \bbrack{P^\rho_b \otimes \lambda_b} \right\}$,
  \item the homology group $H_2(M;\sllrho)$ has a basis  $\left\{
    \bbrack{P^\rho_1\otimes T^2_1}, \ldots, \bbrack{P^\rho_b \otimes T^2_b} \right\}$.
  \end{itemize}
In \cite[Definition 3.21]{Porti:1997}, irreducibility is not required and the second condition is written by using 
twisted cohomology groups.
Since we consider representations near the holonomy representation of $\pi_1(M)$,
 we focus on $\bm{\lambda}$-regularity of irreducible representations in the present article.
\end{remark}

Here we use the same symbol $\lambda_\ell$ and $T^2_\ell$ for lifts of $\lambda_\ell$ and $T^2_\ell$ to the universal cover.

\begin{remark}
  For generic points on the geometric component of $X(M)$, the
  corresponding $\SL$-representations satisfy
  $\bm{\lambda}$-regularity.
\end{remark}
  
We assume the following assumption for $\SL$-representation of
$\pi_1(M)$
\begin{itemize}
\item[$(A_\rho)$] the representation $\rho\colon \pi_1(M) \to \SL$ is
  $\bm{\lambda}$-regular for the multi--longitude $\bm{\lambda}$ determined by $\varphi \colon \pi_1(M) \to \ZZ^n$.
\end{itemize}

\begin{example}[Holonomy representation]\label{ex:holonomy}
  Suppose that $M$ is a hyperbolic three--di\-men\-sio\-nal manifold. Let
  $\rho_0\colon \pi_1(M) \to \SL$ be a lift of the
  discrete and faithful representation of $\pi_1(M)$ to
  $\mathrm{PSL}_2(\CC)$ given by the hyperbolic structure. Porti
  proves~\cite{Porti:1997} that $\rho_0$ satisfies Assumption
  $(A_\rho)$ for any system of homotopically non--trivial curves $(\gamma_1, \ldots,
  \gamma_b)$.
 \end{example}

We mention a relation between the character variety $X(M)$ and the
twisted homology group $H_1(M;\sllrho)$ for $\bm \lambda$-regular representation.  
\begin{remark}
  The twisted cohomology group $H^1(M;\sllrho)$ is the dual space of
  the twisted homology group $H_1(M;\sllrho)$ by the Universal
  Coefficient Theorem.  Following~\cite{TN} and~\cite[Proposition
  3.2.1]{CS:1983} together with $\chi(M)=0$, it is known that, near
  the discrete and faithful representation, the character variety
  $X(M)$ is a complex affine variety with dimension $b$ where $b$ is
  the number of torus boundary components.  In affine varieties, the
  dimension of Zariski tangent space is not less that the one of
  of the variety.  The Zariski tangent space of the character variety
  can be injectively mapped into $H^1(M;\sllrho)$.  We also have 
  $\dim_{\CC} H^1(M;\sllrho) = b$ near the discrete and faithful
  representation, thus the spaces $X(M)$, $T^{Zar}_{\chi_\rho}X(M)$
  and $H^1(M;\sllrho)$ have the same dimension $b$.  This means that
  characters near the discrete and faithful representation are
  smooth points and $H^1(M;\sllrho)$ is identified with the tangent
  space $T_{\chi_\rho}X(M)$ (see also \cite[Chapter 3]{Porti:1997} for
  such identifications).
\end{remark}

\subsection{Acyclicity}
\label{section:acyclicity_complex}

This subsection is devoted to prove the acyclicity of
$C_*(M;\sllmt_\rho)$, \ie, $H_*(M;\sllmt_\rho) = 0$, from the 
assumptions referred to as $(A_M)$, $(A_\varphi)$ and $(A_\rho)$;
\begin{description}
\item[$(A_M)$] the canonical inclusion $i\colon \partial M \hookrightarrow M$ of
  $\partial M$ into $M$ induces an homomorphism
  $i_*\colon H_1(\partial M; \ZZ) \to H_1(M; \ZZ)$ which is onto and
  such that $\mathrm{rank} (i_{|T^2_\ell})_* = 1$ for all boundary
  component $T^2_\ell$ of $\partial M = \bigcup_\ell T^2_\ell$;
\item[$(A_\varphi)$] for all boundary component $T^2_\ell$ of $\partial M$,
  there exist non negative integers $\big(a_1(\ell), \ldots, a_n(\ell)\big) \in \NN^n\setminus\{0\}$
  such that 
  $\varphi(\pi_1(T^2_\ell)) = \langle t_1^{a_1(\ell)} \cdots t_n^{a_n(\ell)} \rangle;$ 
\item[$(A_\rho)$] the $\SL$-representation $\rho$ is sufficiently regular,
  to be more precise $\rho$ is supposed to be $(\lambda_1, \ldots, \lambda_n)$-regular
  where $(\lambda_1, \ldots, \lambda_n)$ is a system of longitudinal curves
  on $\partial M = \bigcup_\ell T^2_\ell$.
\end{description}

\subsubsection{The acyclicity of local system for manifolds with tori--boundary}
\label{section:acyclicity_boundary}

In this section we suppose that the representations $\rho\colon
\pi_1(M) \to \SL$ under consideration satisfy Assumption~$(A_\rho)$.

We prove that the twisted chain complex $C_*(M; \sllmt_\rho)$ is
acyclic in that case.

\begin{proposition}\label{prop:acyclic}
  We have $H_*(M; \sllmt_\rho) = 0$.
\end{proposition}

The proof of Proposition~\ref{prop:acyclic} is based on Milnor's
construction~\cite{Milnor:1968}, but the techniques are rather different,
actually we use the restriction map induces by the inclusion $\partial M \hookrightarrow M$.

Let $\infcover{M}$ denote the infinite cyclic covering of $M$.
We have $\ker (\mathrm{pr}_n \circ \varphi) = \pi_1(\infcover{M})$ where
$\mathrm{pr}_n \colon \ZZ^n \to \ZZ$ denotes the projection by substituting $t_1 = \cdots = t_{n-1} = 1$.
We use the symbol $F$ for the fraction field $\CC(t_1, \ldots, t_{n-1})$.
The action of
$\otAd{\varphi}{\rho^{-1}}$ of $\CC[t_n, t_n^{-1}] \otimes_{\CC} \msll{F}_\rho$
is given by the tensor product of $\mathrm{pr}_n \circ \varphi$ and
$\otAd{(\mathrm{pr}_{1, \ldots, n-1}) \circ \varphi}{\rho^{-1}}$
where $\mathrm{pr}_{1, \ldots, n-1}$ denotes the projection by substituting $t_n = 1$.
Moreover, one can observe that under the inclusion $\ZZ[t_n, t_n^{-1}] \to \CC[t_n, t_n^{-1}]$,
$$
C_*(\infcover{M}; \msll{F}_\rho) \simeq 
C_*(M; \CC[t_n, t_n^{-1}] \otimes_\CC \msll{F}_\rho).
$$ 
The Milnor sequence:
\begin{equation}\label{eqn:MilnorSq}
    0 \to
    C_*(\infcover{M}; \msll{F}_\rho) \xrightarrow{t_n - 1} 
    C_*(\infcover{M}; \msll{F}_\rho) \xrightarrow{t_n = 1} 
    C_*(M; \msll{F}) \to
    0
\end{equation}
 induces the
long exact sequence in twisted homology:
\begin{align}\label{eqn:MilnorSq2}
    0 &\to 
    H_2(\infcover{M}; \msll{F}_\rho) \xrightarrow{t_n - 1} 
    H_2(\infcover{M}; \msll{F}_\rho) \xrightarrow{t_n = 1} 
    H_2(M; \msll{F}_\rho)  \\
    & \xrightarrow{\delta} 
    H_1(\infcover{M}; \msll{F}_\rho) \xrightarrow{t_n - 1}  
    H_1(\infcover{M}; \msll{F}_\rho) \xrightarrow{t_n = 1} 
    H_1(M; \msll{F}_\rho)  \to
    0. \nonumber
\end{align}

Proposition~\ref{prop:acyclic} is a consequence of the following lemma.
\begin{lemma}\label{lemma:free_bound}
  Let $F$ be the fraction field $\CC(t_1, \ldots, t_{n-1})$.  The
  homology group $$H_*(M; \CC[t_n, t_n^{-1}] \otimes_\CC \msll{F}_\rho)$$ has
  no free part.
\end{lemma}
\begin{proof}[Proof]
  The proof is by induction on the number $n$ of variables $t_1, \ldots, t_n$. 
  The first step is to prove the lemma in the case of a
  single variable $t$.

\medskip

\noindent \textbf{(i)}. 
We prove that \emph{$H_*(M; \CC[t, t^{-1}] \otimes_\CC \sllrho)$ has no free part.}
  
Let $A_\ell$ be the annulus in $\infcover{M}$ over $T_\ell^2$ and $b$ the
number of the components of $\bord M$. The Milnor sequence for the
boundary $\partial M$ induces the first line of the following
commutative diagram, the second one is the exact
sequence of Equation~(\ref{eqn:MilnorSq2}), and the diagram is commutative because
all the constructions are natural:
$$
\xymatrix@-.7pc{ 
  0 \ar[r] & 
  \displaystyle{\mathop{\oplus}_{\ell=1}^{b} H_2(T_\ell^2; \sllrho)} \ar[r]^-{\delta_\bord} \ar[d]^-\simeq & 
  \displaystyle{\mathop{\oplus}_{\ell=1}^{b} H_1(A_\ell; \sllrho)}  \ar[r]^{t - 1} \ar[d] & 
  \displaystyle{\mathop{\oplus}_{\ell=1}^{b} H_1(A_\ell; \sllrho)} \ar[r]^-{t = 1} \ar[d] & 
  \cdots \\
  \cdots \ar[r]^-{t=1} & 
  H_2(M; \sllrho) \ar[r]^-{\delta} &
  H_1(\infcover{M}; \sllrho) \ar[r]^{t - 1} & 
  H_1(\infcover{M}; \sllrho) \ar[r]^-{t = 1} & 
  \cdots }
$$

The proof is by contradiction, so we make the following hypothesis:

\begin{center} 
{\noindent ($\mathcal{H}$) \quad
    \emph{
      $H_1(\infcover{M}; \sllrho) \simeq H_1(M; \CC[t, t^{-1}] \otimes_\CC \sllrho)$ 
       has a free part of rank $r>0$}. 
     }
\end{center}

\noindent First observe that (because $\chi(M) = 0$ and $H_0(\infcover{M}; \sllrho) = 0$): 
\begin{equation}\label{eqn:rank}
  \mathrm{rk} \, H_2(\infcover{M}; \sllrho) = 
  \mathrm{rk} \, H_1(\infcover{M}; \sllrho) = r.
\end{equation}
 
Our proof is as follows and based on the following technical claim.
\begin{claim}\label{claim:map_Delta}
  The map 
  $$\delta_{\bord} : 
    \bigoplus_{\ell=1}^{b} H_2(T_\ell^2; \sllrho) \to
    \bigoplus_{\ell=1}^{b} H_1(A_\ell; \sllrho)$$
  in the previous diagram is
  non--trivial, and moreover for all $\ell$, 
  $\delta_{\bord}\left( \bbrack{P_\ell^\rho \otimes T^2_\ell} \right)$  
  is non zero.
\end{claim}
\begin{proof}[Proof of the Claim]
  Let $P_\ell^\rho$ $(1 \myleq \ell \myleq b)$ be the chosen invariant
  vector in $\sll$.  Since $\bord M = \cup_{\ell=1}^b T^2_\ell$, we
  have $H_0(\partial M; \sllrho) = \bigoplus_{\ell=1}^{b}
  H_0(T^2_\ell; \sllrho).$ It follows from our assumptions that
  $H_2(M; \sllrho)$ is generated by the vectors $\bbrack{P_\ell^\rho
    \otimes T^2_\ell}$ $(1 \myleq \ell \myleq b)$ and $H_1(M;
  \sllrho)$ is also generated by the vectors $\bbrack{P^\rho_\ell
    \otimes \lambda_\ell}$ $(1 \myleq \ell \myleq b)$.  The space
  $H_2(T_\ell^2; \sllrho)$ is generated by $\bbrack{P_\ell^\rho
    \otimes T^2_\ell}$ and we have:
  \begin{equation}\label{eqn:partial}
    \delta_{\partial}\left( 
      \bbrack{P_\ell^\rho \otimes T^2_\ell} 
     \right) = 
       (1 + t + \cdots + t^{a_\ell - 1})
       \bbrack{P^\rho_\ell \otimes \lambda_\ell},
  \end{equation}
  where $\varphi(\mu_\ell) = t^{a_\ell}$, $a_\ell>0$.  Next it is easy to observe
  that each $(\sum_{k=0}^{a_\ell - 1} t^k) \bbrack{P^\rho_\ell \otimes \lambda_\ell}$ 
  is a non zero element in $H_1(\infcover{M}; \sllrho)$,
  because its image by the map $(t=1)$ is 
  $a_\ell \bbrack{P^\rho_\ell \otimes \lambda_\ell}$, 
  which is non zero in $H_1(M; \sllrho)$. 
  This proves that $\delta_{\bord M} \ne 0$.
\end{proof}

Using the exactness of the Milnor sequence~(\ref{eqn:MilnorSq2}) and
 Equation~(\ref{eqn:rank}), we know that
$$
\im (t=1) \simeq 
  H_2(\infcover{M}; \sllrho)/{\left( {(t-1)H_2(\infcover{M}; \sllrho} \right)} \ne 0.
$$ 
Since $\im (t=1) = \ker \delta$, we deduce that $\ker \delta \ne 0$. 
Let $\xi$ be a non zero element in $\ker \delta$ and write it in
$H_2(M; \sllrho)$ as follows:
\[
\xi = \sum_{\ell=1}^{b} b_\ell \bbrack{P^\rho_\ell \otimes T^2_\ell}.
\]
One has (see Equation~(\ref{eqn:partial})):
\[
\delta_{\partial}(\xi) = 
\sum_{\ell=1}^{b} b_\ell
(1 + t + \cdots + t^{a_\ell - 1})
\bbrack{P^\rho_\ell \otimes \lambda_\ell}
\]
and thus,
\[
(t=1) \circ \delta_{\partial} (\xi) = 
  \sum_{\ell=1}^{b} a_\ell b_\ell \bbrack{P^\rho_\ell \otimes \lambda_\ell}.
\]
Now we prove by contradiction that $\delta_{\partial}(\xi)$ is non
zero. If $\delta_{\partial}(\xi) = 0$, then its image by $(t=1)$ is
also zero, so $a_\ell = 0$ or $b_\ell = 0$, for all $i$. Since $a_\ell > 0$,
for all $i$, we deduce that $b_\ell = 0$, for all $i$. So that, $\xi =
0$, which is a contradiction and thus $\delta_{\partial}(\xi) \ne 0$.

With our assumption we have 
$(t=1)\circ \delta_\partial(\xi) = 
  \sum_{\ell=1}^{b} a_\ell b_\ell \bbrack{P^\rho_\ell \otimes \lambda_\ell}$, 
and this element is non zero in $H_1(\infcover{M}; \sllrho)$ as we seen. 
But this is in contradiction with the fact that $\xi$ is chosen in $\ker \delta_\partial$, 
so that the
hypothesis ($\mathcal{H}$) on the free part of 
$H_1(M; \CC[t, t^{-1}] \otimes_\CC \sllrho)$ is absurd and 
proves Lemma~\ref{lemma:free_bound} in the case of a single variable.

\medskip 

\noindent \textbf{(ii)}. 
Now we 
\emph{finish the proof by induction on  the number of variables}, 
and suppose that
$$H_*(M; \CC[t_{n-1}, t_{n-1}^{-1}] \otimes_\CC \msllrho{t_1, \ldots, t_{n-2}})$$ 
has no free part. 
Thus, $H_*(M; \msllrho{F})$ vanishes and 
the long exact sequence~(\ref{eqn:MilnorSq2}) induces
the isomorphism:
\[
  0 \to 
  H_*(\infcover{M}; \msllrho{F}) \xrightarrow[\simeq]{t_n - 1}  
  H_*(\infcover{M}; \msllrho{F}) \to
  0.
\]
As a conclusion, $C_*(M; \msllrho{F})$ is acyclic which proves Lemma~\ref{lemma:free_bound}.

\end{proof}

Thus, the twisted complex $C_*(M; \sllmt_\rho)$ is acyclic and the torsion is
well--defined (even its sign if we provide $M$ with its natural
homology orientation, see e.g.~\cite{Turaev:2002}):
\begin{equation}\label{def:polytorsion}
  \PolyTors{M}{\otAd{\varphi}{\rho}}(t_1, \ldots, t_n)
  =
  \tau_0 \cdot 
  \Tor{C_*(M;\sllmt_\rho)}{\cbasis{*}}{\emptyset} 
    \in \CC(t_1, \ldots, t_n).
\end{equation}
Here $ \tau_0 = \mathrm{sgn} (\Tor{C_*(M; \IR)}{\cbasis{*}_\IR}{\hbasis{*}_\IR})$. 

\section{Examples of computations}
\label{section:exemples}

We compute the polynomial torsions for the figure knot exterior and
the Whitehead link exterior by using Fox differential calculus as
shown in~\cite{Milnor:1968,Kitano, KL}.

First, we consider the
Jacobian matrix using Fox free differential calculus associated to a
Wirtinger presentations of a link group.  
To express the polynomial torsion, we need a square minor in the
Jacobian matrix.  Since the number of relations in a Wirtinger
presentation is one less than that of generators, we have a square
minor in the Jacobian matrix by dropping one column.  For an
$\SL$-representation $\rho$ of the link group, when we replace each
element of the link group in the square minor of the Jacobian matrix
by the $3 \times 3$ matrix derived from the action of
$\otAd{\varphi}{\rho}$, we obtain a large matrix whose entries are
Laurent polynomials with coefficients in $\CC$.  Then we can express
the polynomial torsion for the link exterior as the rational function
whose numerator is the determinant of the square minor replaced each
component with $\otAd{\varphi}{\rho}$.  The denominator of the
polynomial torsion is the characteristic polynomial of the
$\mathrm{SL}_3(\CC)$-element given by the generator corresponding to
the dropped column from the Jacobian matrix.  It remains a problem to
construct $\SL$-representations of link group.  However we can find
explicit constructions for the figure eight knot
in~\cite{KirkKlassenMathAnn:ChernSimon} and for the Whitehead link
in~\cite{HLM92:characterization}.

\subsection{The figure eight knot  exterior}
We consider the figure eight knot $K$ as in Figure~\ref{fig:figure8_knot}.
The knot group $\pi_1(E_K)$ is expressed as 
$$
E_K = \langle x, y \,|\, [x^{-1}, y]x = y[x^{-1}, y] \rangle 
$$
where $E_K$ is the complement an open tubular neighbourhood $N(K)$ of $K$ in $S^3$. 
\begin{figure}[!h]
  \begin{center}
    \includegraphics[scale=.6]{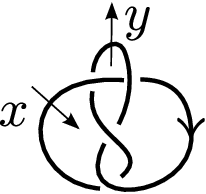}
  \end{center}
  \caption{The figure eight knot}
  \label{fig:figure8_knot}
\end{figure}
The following correspondences give an $\SL$-representation $\rho_{\sqrt{s}, u}$ of $\pi_1(E_K)$
\begin{equation}
x \mapsto 
\begin{pmatrix}
  \sqrt{s} & 1/\sqrt{s} \\
  0 & 1/\sqrt{s}
\end{pmatrix},
\quad 
y \mapsto 
\begin{pmatrix}
  \sqrt{s} & 0 \\
  -u\sqrt{s} & 1/\sqrt{s}
\end{pmatrix}
\end{equation}
when the pair $(s, u)$ is a root of $\phi(s, u) = u^2 + (3 - (s+1/s))(u+1)$.

By using formula in~\cite{Milnor:1968,Kitano, KL},
the polynomial torsion is expressed as 
\[
\PolyTors{E_K}{\otAd{\varphi}{\rho_{\sqrt{s}, u}}}(t)%
= \tau_0 \cdot %
\frac{%
  \det \Phi \left(%
    \frac{\partial}{\partial y}\,%
    [x^{-1}, y]x[y, x^{-1}]y^{-1}%
  \right)%
}{%
  \det \Phi(x-1)%
}
\]
where $\Phi$ is the linear extension of $\otAd{\varphi}{\rho_{\sqrt{s}, u}}$ on $\ZZ[\pi_1(E_L)]$.
The Fox differential turns into 
\[
\frac{\partial}{\partial b}%
\left(%
  [x^{-1}, y]x[y, x^{-1}]y^{-1}%
\right)
= x^{-1} - x^{-1}yxy^{-1} + x^{-1}yxy^{-1}x - yx^{-1} -1.
\]
Therefore the numerator of $\PolyTors{E_K}{\otAd{\varphi}{\rho_{\sqrt{s}, u}}}(t)$
turns out
$$ \tau_0 \cdot \det \Phi(x^{-1} - x^{-1}yxy^{-1} + x^{-1}yxy^{-1}x - yx^{-1} -1).$$ 

When we choose the basis $\{E, H, F\}$ in $\sll$, the adjoint actions
$Ad_{\rho(x)^{-1}}$ and $Ad_{\rho(y)^{-1}}$ are represented by the following 
upper and lower triangular matrices:
\[
Ad_{\rho(x)^{-1}}=
  \begin{pmatrix}
    1/s & 2/s & -1/s \\
    0 & 1 & -1 \\
    0 & 0 & s
  \end{pmatrix}
\quad
Ad_{\rho(y)^{-1}}=
  \begin{pmatrix}
    1/s & 0 & 0 \\
    -u & 1 & 0 \\
    -su^2 & 2su & s
  \end{pmatrix}.
\]

Calculating the determinant and reducing with the equation $\phi(s, u)=0$, we can obtain the polynomial as
\begin{align}\label{eqn:PolyTorsion_figure8}
  \PolyTors{E_K}{\otAd{\varphi}{\rho_{\sqrt{s}, u}}}(t)
  &=
  \frac{%
    \tau_0 \cdot \frac{1}{t^3} \cdot (t-1)^2(t-s)(t-1/s)(t^2 - (2 s + 2/s +1)t +1)%
  }{%
    (t-s)(t-1)(t-1/s)%
  } \nonumber \\
  &= \tau_0 \cdot 
    \frac{1}{t^3} 
    (t-1)\left( t^2 - (2I_x^2 -3)t +1\right)
\end{align}
where $I_x = \sqrt{s} + 1/\sqrt{s}$ is the trace function of the meridian $x$.
Note that we can use $I_y$ instead of $I_x$ since all generators in a Wirtinger presentation are conjugate.

\subsection{The Whitehead link exterior}
\label{section:polynomial_torion_Whiteheadlink}
Let $L$ be the Whitehead link and choose
the following Wirtinger presentation of the Whitehead link group:
\[
\pi_1(E_{L}) = \langle a, b \,|\, awa^{-1}w^{-1} \rangle
\text{ where } w = bab^{-1}a^{-1}b^{-1}ab.
\]
\begin{figure}[h]
  \centering
  \includegraphics[scale=.5]{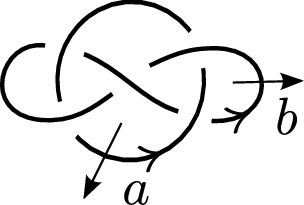}
  \caption{The Whitehead link with meridians}
  \label{fig:WhiteheadLinkMeridians}
\end{figure}

Hilden, Lozano and Montesionos has shown an
explicit description of the character variety of the Whitehead link
group in~\cite{HLM92:characterization}.  Trace functions play the role
of local coordinates in this description.   We set $x$, $y$ and $v$ as
$x=I_a$, $y=I_b$ and $v=I_{ab}$, where $I_\gamma \colon X(E_L) \to \CC$ is again given by $I_\gamma (\chi) = \chi(\gamma)$.  Then the character variety of
$X(E_L)$ is expressed as
\[
X(E_L) = \{(x, y, v)\in \CC^3 \,|\, p(x, y, v) q(x, y, v) = 0\},
\]
where 
\begin{align*}
  p(x, y, v) &= xy - (x^2 + y^2 - 2)v + xyv^2 - v^3,\\
  q(x, y, v) &= x^2 + y^2 + v^2 -xyv - 4.
\end{align*}
The component of irreducible characters is given by 
\[ 
X^{\mathrm{irr}}(E_L) =\{(x, y, v)\in \CC^3 \,|\, p(x, y, v) = 0\} \setminus R
\]
where $R = \{x = \pm 2, v= \pm y\} \cup \{y = \pm 2, v = \pm x\}$.  We
consider an $\SL$-representation $\rho$ of $\pi_1(E_L)$ whose
character is contained in $X^{\mathrm{irr}}(E_L)$.  

Let $\rho \colon \pi_1(E_L) \to \SL$ be an irreducible representation and consider  
the abelianization $\varphi$ of $\pi_1(E_L)$: $\varphi\colon \pi_1(E_L) \to \ZZ^2 = \langle t_1, t_2 \,|\, t_1t_2 = t_2t_1 \rangle$ defined by $\varphi(a) = t_1$ and $\varphi(b)
= t_2$.  

Using irreducibility of $\rho$, we can suppose that the pair of
$\rho(a)$ and $\rho(b)$ are expressed as (after taking conjugation with eigenvectors of $\rho(a)$ and $\rho(b)$ if necessary):
\[
\rho(a) = 
\left(
  \begin{array}{cc}
    \alpha & 1 \\
    0 & 1/\alpha
  \end{array}
\right)
\quad
\rho(b) = 
\left(
  \begin{array}{cc}
    \beta & 0 \\
    \gamma & 1/\beta
  \end{array}
\right).
\]

Note that we have set the local coordinates $(x, y, v)$ as 
$x = \alpha + \alpha^{-1}$, $y = \beta + \beta^{-1}$ and 
$v = \gamma + \alpha\beta + \alpha^{-1}\beta^{-1}$. 

By using formula in~\cite{Milnor:1968,Kitano, KL},
the polynomial torsion is expressed as 
\[
\PolyTors{E_L}{\otAd{\varphi}{\rho}}(t_1, t_2)
= 
\tau_0 \cdot 
\frac{\det \Phi \left( \frac{\partial}{\partial b}\, awa^{-1}w^{-1}  \right)}{\det \Phi(a-1)}
\]
where $\Phi$ is the linear extension of $\otAd{\varphi}{\rho^{-1}}$ on $\ZZ[\pi_1(E_L)]$.
The Fox differential turns into 
\[
\frac{\partial}{\partial b}\left( {awa^{-1}w^{-1} }\right)
= (a-1)(1 - bab^{-1} - bab^{-1}a^{-1}b^{-1} + bab^{-1}a^{-1}b^{-1}a).
\]
Therefore the polynomial torsion $\PolyTors{E_L}{\otAd{\varphi}{\rho}}(t_1, t_2)$
turns out 
$$\PolyTors{E_L}{\otAd{\varphi}{\rho}}(t_1, t_2) 
= \tau_0 \cdot \det \Phi(1 - bab^{-1} - bab^{-1}a^{-1}b^{-1} +
bab^{-1}a^{-1}b^{-1}a).$$ 
So that, with $\det Ad _{\rho (a)^{-1}} =
\det Ad_{\rho(b)^{-1}} = 1$ in mind,
\begin{align*}
  \PolyTors{E_L}{\otAd{\varphi}{\rho}}(t_1, t_2)
  &= \tau_0 \cdot 
     \det \Phi(1 - bab^{-1} - bab^{-1}a^{-1}b^{-1} + bab^{-1}a^{-1}b^{-1}a)\\
  &= \tau_0 \cdot 
     \det \Phi\left( (bab^{-1}a^{-1})(aba^{-1}b^{-1} - a - b^{-1} + b^{-1}a ) \right)\\
  &= \tau_0 \cdot  
     \det \Phi(aba^{-1}b^{-1} - a - b^{-1} + b^{-1}a ).
\end{align*}

When we choose the basis $\{E, H, F\}$ in $\sll$, the adjoint actions
$Ad_{\rho(a)^{-1}}$ and $Ad_{\rho(b)^{-1}}$ are represented by the following 
upper and lower triangular matrices
\[
Ad_{\rho(a)^{-1}}=
\left(
  \begin{array}{ccc}
    1/\alpha^2 & 2/\alpha & -1 \\
    0 & 1 & -\alpha \\
    0 & 0 & \alpha^2
  \end{array}
\right)\, \hbox{and}\,\,
Ad_{\rho(b)^{-1}}=
\left(
  \begin{array}{ccc}
    1/\beta^2 & 0 & 0 \\
    \gamma/\beta & 1 & 0 \\
    -\gamma^2 & -2\beta\gamma & \beta^2
  \end{array}
\right).
\]

Calculating the determinant and reducing the degree of $\gamma$ by
using the following identity $p(\alpha + \alpha^{-1}, \beta + \beta^{-1}, \gamma + \alpha
\beta + \alpha^{-1}\beta^{-1})=0$, we can see that the polynomial
torsion is expressed as
\begin{align}\label{eqn:PolyTorsion_WL}
  &\PolyTors{E_L}{\otAd{\varphi}{\rho}}(t_1, t_2)\\
  &= \tau_0 \cdot 
    \frac{(t_1-1)(t_2-1)}{t_2^3} 
    \left( -2xyv t_1t_2 +
    x^2 t_1(t_2+1)^2 + y^2 (t_1+1)^2 t_2 - (t_1+1)^2 (t_2+1)^2
    \right). \nonumber
\end{align}

\section{Change of coefficients}
\label{section:naturality}

The aim of this section is to give some notation and explanation
about the reduction of variables (for more details,
see~\cite{Milnor:1966}).

Let $\varphi$ be a surjective homomorphism of $\pi_1(M)$ onto
$\ZZ^n=\langle t_1, \ldots, t_n \,|\, t_i t_j = t_j t_i \rangle$.  
We let $h_{(a_1, \ldots, a_n)}$ be the homomorphism of $\ZZ^n$ into
$\ZZ=\langle t \rangle$, given by 
\[
h_{(a_1, \ldots, a_n)} (t_1, \ldots, t_n)
=
(t^{a_1}, \ldots, t^{a_n})
\]
where each $a_\ell$ is a positive
integer. We use the notation $\varphi_{(a_1, \ldots, a_n)}$ for the composition of
$\varphi$ and $h_{(a_1, \ldots, a_n)}$:
\begin{equation}\label{diagram:Zm_to_Z}
  \xymatrix@R=20pt{
    \pi_1(M) \ar[r]^-\varphi \ar[dr]_{\varphi_{(a_1, \ldots, a_n)}} & \ZZ^n \ar[d]^{h_{(a_1, \ldots, a_n)}} \\  & \ZZ}
\end{equation}

Observe that $\varphi_{(\alpha_1, \ldots,
  \alpha_n)}$ is onto if and only if G.C.D $(a_1, \ldots, a_n)$
is $1$.  Moreover $\varphi_{(\alpha_1, \ldots, \alpha_n)}$
satisfies the condition $(A_\varphi)$ since each $a_\ell$ is positive.

Later in this paper, we often make a reduction of several variables
into one variable.
We let $\varphi$ be a surjective homomorphism of $\pi_1(M)$ onto
$\ZZ^n$ satisfying $(A_\varphi)$ and $\rho$ be an $\SL$-representation
satisfying $(A_\rho)$.  We choose relatively prime positive integers
$(a_1, \ldots, a_n)$ and let $\varphi_{(a_1, \ldots, a_n)}$ be the
composition of $\varphi$ and $h_{(a_1, \ldots, a_n)}$.  From
Section~\ref{section:acyclicity_complex}, the Reidemeister torsion
$\PolyTors{M}{\otAd{\varphi}{\rho}}(t_1, \ldots, t_n)$ and
$\PolyTors{M}{\otAd{\varphi_{(a_1, \ldots, a_n)}}{\rho}}(t)$ are
defined for both the abelian homomorphisms $\varphi$ and
$\varphi_{(a_1, \ldots, a_n)}$.
The following result (with sign) is a consequence of the definitions. 

\begin{proposition}[\cite{Milnor:1966}]\label{cor:naturality}
  One has the following formula:
  \[
  \PolyTors{M}{\otAd{\varphi}{\rho}}(t^{a_1}, \ldots, t^{a_n}) =
  \PolyTors{M}{\otAd{\varphi_{(a_1, \ldots, a_n)}}{\rho}}(t).
  \]
\end{proposition}

\section{A derivative formula}
\label{section:derform}

We prove a relation between the polynomial torsion and 
the non--abelian Reidemeister torsion. 
More precisely, we prove that
the non--abelian Reidemeister is a sort of ``differential coefficient''
associated to the polynomial torsion.

We review the definition of the non--abelian Reidemeister torsion (for
more details, we refer to~\cite[Chap.~3]{Porti:1997}).
\begin{definition}\label{def:torsion}
  Let $M$ be a compact hyperbolic three--dimensional manifold whose
  boundary is the disjoint union of $b$ tori $\partial M = \cup_{\ell = 1}^b T^2_\ell$. Consider an
  $\SL$-representation $\rho \colon \pi_1(M) \to \SL$ which is
  $\bm{\lambda}$-regular for a set of closed loops $\bm{\lambda} =
  \{\lambda_\ell \subset T^2_\ell\,|\, 1 \myleq \ell \myleq b \}$.
  The non--abelian Reidemeister torsion
  $\RTors{M}{\smallbm{\lambda}}(\rho)$ is defined to be the sign--refined
  Reidemeister torsion for $C_*(M; \sllrho)$ and the basis
  $$\hbasis{*}_{\smallbm{\lambda}} = \big\{ 
    \bbrack{P^\rho_1 \otimes T^2_1},
      \ldots, 
    \bbrack{P^\rho_b \otimes T^2_b}, 
    \bbrack{P^\rho_1 \otimes \lambda_1},
      \ldots, 
    \bbrack{P^\rho_b \otimes \lambda_b}
                              \big\}$$ as
  \[
  \RTors{M}{\smallbm{\lambda}}(\rho) = \tau_0 \cdot \Tor{C_*(M;
    \sllrho)}{\cbasis{*}}{\hbasis{*}_{\smallbm{\lambda}}}.
  \]
\end{definition} 

\subsection{Bridge from the polynomial torsion to the non--abelian
  Reidemeister torsion}

The following theorem proves that the non--abelian Reidemeister
torsion can be deduced from the polynomial torsion
$\PolyTors{M}{\otAd{\varphi}{\rho}}(t)$ (see~\cite{YY:Fourier} for the
case of knots).

In this section, we suppose that $M$ is a compact hyperbolic three--dimensional manifold satisfying assumption 
$(A_M)$, $\varphi\colon \pi_1(M) \to \ZZ = \langle t \rangle$ is a
surjective homomorphism which satisfies assumption $(A_\varphi)$
and that $\rho$ satisfies assumption $(A_\rho)$.  We equip the three--manifold
$M$ with a distinguished homology orientation.

\begin{theorem}\label{theorem:polyvstorsion}
  The following equality holds:
  \begin{equation}\label{eqn:limit_formula}
    \lim_{t \to 1} 
    \frac{
      \PolyTors{M}{\otAd{\varphi}{\rho}}(t)
    }{
      \prod_{\ell = 1}^b (t^{a_\ell}-1)
    }
    = (-1)^b \cdot  \RTors{M}{\smallbm{\lambda}}(\rho)
    \quad \left(\bm{\lambda}=(\lambda_1, \ldots, \lambda_b)\right), 
  \end{equation}
  where $\varphi(\pi_1(T^2_\ell)) = \langle t^{a_\ell}\rangle$, $a_\ell
  \in \ZZ_{>0}$, and $b$ is the number of components of $\partial M$.
\end{theorem}

Before proving this result, we give a couple of remarks.

\begin{remark}
  Using Theorems~\ref{theorem:poly} \& \ref{theorem:polyvstorsion} one
  can observe that if $\varphi\colon \pi_1(M) \to \ZZ = \langle t
  \rangle$ satisfies assumption $(A_\varphi)$, then $(t-1)^b$ divides
  $\PolyTors{M}{\otAd{\varphi}{\rho}}(t)$.
\end{remark}

\begin{remark}[The multivariable case]
  Here we suppose that $\varphi\colon \pi_1(M) \to \ZZ^n$ where $\ZZ^n
  = \langle t_1, \ldots, t_n \, |\, t_i t_j = t_j t_i, \; \forall i,j
  \rangle$.
\end{remark}
	
\begin{corollary}\label{cor:DerFormula_one _variable}
  We have the following identity:
  \begin{equation}
    \lim_{t_1, \ldots, t_n \to 1} 
    \frac{
      \PolyTors{M}{\otAd{\varphi}{\rho}}(t_1, \ldots, t_n)
    }{
      \prod_{\ell = 1}^b (t_1^{a_1^{(\ell)}} \cdots t_n^{a_n^{(\ell)}} - 1)
    } 
    = (-1)^b \cdot \RTors{M}{\smallbm{\lambda}}(\rho),
  \end{equation}
  where $\varphi(\pi_1(T_\ell^2)) = \langle t_1^{a_1^{(\ell)}} \cdots
  t_n^{a_n^{(\ell)}} \rangle$, $a_1^{(\ell)}, \ldots, a_n^{(\ell)} \in \ZZ_{>0}$.
\end{corollary}

\subsection{Proof of Theorem~\ref{theorem:polyvstorsion}}

We begin by introducing the complexes needed in the proof and some
notation. 

Let $C_* = C_*(M; \sllrho)$ and $C_*(t) = C_*(M; \sllrhot)$.  We
define a pair of complexes $(C'_*, C'_*(t))$ as follows.  The complex
$C'_*$ is defined as a subchain complex of $C_*$ which is a lift of
the homology group $H^\rho_*(M) = H_*(M; \sllrho)$.  This is a
``degenerated complex'' in the sense that the boundary operators are
all zero.  More precisely, $C'_3 = C'_0 = 0$ for conventions, $C'_2$
is spanned (over $\CC$) by 
$\{ P^\rho_\ell \otimes T^2_\ell \,|\, 1 \myleq \ell \myleq b \}$, 
and $C'_1$ is spanned (over $\CC$)
by 
$\{P^\rho_\ell \otimes \lambda_\ell \,|\, 1 \myleq \ell \myleq b \}$.  
Similarly as $C_*(t)$, we define $C'_*(t)$ by changing the coefficient from $\sll$
to $\sllt = \CC(t) \otimes \sll$, in particular
$C'_1(t)$ is spanned by 
$\{1 \otimes P^\rho_\ell \otimes \lambda_\ell \,|\, 1 \myleq \ell \myleq b \}$ and $C'_2(t)$ is spanned by 
$\{1 \otimes P^\rho_\ell \otimes T^2_\ell \,|\, 1 \myleq \ell \myleq b \}$.  
Observe that
$C'_*(t)$ is a subchain complex of $C_*(t)$.  More precisely, one has:
\begin{equation}\label{eqn:c'(t)}
  C'_*(t) = 
    0 \to 
    C'_3(t) \to C'_2(t) \xrightarrow{\partial'} 
    C'_1(t) \to C'_0(t) \to
    0
\end{equation}
where the boundary operator $\partial'$ works as follows:
$$
  \partial' \colon (1 \otimes P^\rho_\ell) \otimes T^2_\ell
  \mapsto (t^{a_i} - 1) \cdot (1 \otimes P^\rho_\ell) \otimes
    \lambda_\ell.
$$ 
Finally, we define $C''_*$ as the quotient complex $C_* / C'_*$ and
$C''_*(t) = C_*(t) / C'_*(t)$.  Hence we have the two following exact
sequence of complexes:
\begin{equation}\label{eqn:compsanst}
  0 \to C'_* \to C_* \to C''_* \to 0.
\end{equation}
\begin{equation}\label{eqn:comp}
  0 \to C'_*(t) \to C_*(t) \to C''_*(t) \to 0.
\end{equation}

As we already observe, the complexes $C'_*$, $C_*$ are not acyclic and
we completely know their homology groups.  The homology of $C''_*$ is
given in the following claim.
\begin{lemma}
  The complex $C''_*$ is acyclic.
\end{lemma}
\begin{proof}[Proof of the claim]
  Write down the long exact sequence in homology associated to the
  short exact sequence~(\ref{eqn:compsanst}) :
  \[
   \cdots \to H_i(C'_*) \to H_i(C_*) \to
    H_i(C''_*) \to H_{i-1}(C'_*) \to H_{i-1}(C_*) \to
    \cdots.
  \]
  Observe that $H_i(C'_*) \simeq H_i(C_*)$ by definition of $C'_*$.
  Thus $H_i(C''_*) = 0$.
\end{proof}
\begin{remark}
  We think of $C''_*$ as the original complex $C_*$ in which we have
  ``killed'' the homology.
\end{remark}

The homology groups of the complexes $C'_*(t)$, $C_*(t)$ and
$C''_*(t)$ are given in the following claim.
\begin{lemma}\label{lemma:acyclicity_cpxes}
  The complexes $C'_*(t)$, $C_*(t)$ and $C''_*(t)$ are acyclic.
\end{lemma}
\begin{proof}[Proof of the claim]
  According to Proposition~\ref{prop:acyclic}, $C_*(t) = C_*(M;
  \sllrhot)$ is acyclic.  One can observe that the map $\partial'$ is
  invertible, thus $C'_*(t)$ is acyclic.  And finally $C''_*(t) =
  C_*(t) / C'_*(t)$ is also acyclic (as a quotient of two acyclic
  complexes).
\end{proof}

We endow the complexes in Sequence~(\ref{eqn:comp}) with
\emph{compatible bases} in order to compute the torsions.  From the
definition, $C'_*$ is endowed with a distinguished basis
$\cbasis[c']{*}$ given by
\[
\left\{ 
  P^\rho_1 \otimes T^2_1, 
    \ldots, 
  P^\rho_b \otimes T^2_b, 
  P^\rho_1 \otimes \lambda_1, 
    \ldots,
  P^\rho_b \otimes \lambda_b
\right\}
\]
and we equip $C'_*(t)$ with the corresponding distinguished basis $1
\otimes \cbasis[c']{*}$ improperly denoted again for simplicity
$\cbasis[c']{*}$.  Similarly, we endowed the quotient $C''_* = C_* /
C'_*$ with a distinguished basis $\cbasis[c'']{*}$, and the same for
$C''_*(t)$.  Using the exact sequence~(\ref{eqn:comp}), we finally
endowed $C_*(t)$ with the compatible basis $\cbasis[c']{*} \cup
\cbasis[c'']{*}$ obtained by lifting and concatenation (here again our
notation is improper).  Note that this last basis is different from
the distinguished geometric basis $\cbasis{*}$ of $C_*(t) =
C_*(M; \sllrhot)$ described in
Subsection~\ref{section:def_polynomial_torsion}.

From now on, we write $\Tor{C_*(t)}{\cbasis{*}}{\emptyset}$
(resp. $\Tor{C'_*(t)}{\cbasis[c']{*}}{\emptyset}$,
$\Tor{C''_*(t)}{\cbasis[c'']{*}}{\emptyset}$) for the Reidemeister
torsion of $C_*(t)$ (resp. $C'_*(t)$, $C''_*(t)$) computed in the
basis $\cbasis[c']{*} \cup \cbasis[c'']{*}$ (resp. $\cbasis[c']{*}$,
$\cbasis[c'']{*}$); whereas we write $\Tor{C_*(M;
  \sllrhot)}{\cbasis{*}}{\emptyset}$ for the torsion of $C_*(t) =
C_*(M; \sllrhot)$ but computed in the geometric basis $\cbasis{*}$.
Using the basis change formula (see~\cite[Proposition
0.2]{Porti:1997}), we have:
\[
\Tor{C_*(M; \sllrhot)}{\cbasis{*}}{\emptyset} =
\Tor{C_*(t)}{\cbasis[c']{*} \cup \cbasis[c'']{*}}{\emptyset} \cdot
\prod_i [ \cbasis{i}/\cbasis[c']{i} \cup \cbasis[c'']{i} ]^{(-1)^{i}}.
\]

Hence from the definition of the polynomial torsion, we have:
\begin{align}
  \PolyTors{M}{\otAd{\varphi}{\rho}}(t) 
  &= \tau_0 \cdot
  \Tor{C_*(M; \sllrhot)}{\cbasis{*}}{\emptyset} \nonumber \\
  &= \prod_i [ \cbasis{i} / \cbasis[c']{i} \cup \cbasis[c'']{i}
  ]^{(-1)^{i}} \cdot \tau_0 \cdot \Tor{C_*(t)}{\cbasis[c']{*} \cup
    \cbasis[c'']{*}}{\emptyset}    \label{eqn:DE}
\end{align}
where $\tau_0 = \mathrm{sgn} \left(
  \Tor{C_*(M;\IR)}{\cbasis{*}_\IR}{\hbasis{*}_\IR} \right)$, see
Section~\ref{section:def_polynomial_torsion}.

Applying the Multiplicativity Lemma to the exact
sequence in Equation~(\ref{eqn:comp}) we get:
\begin{equation}\label{eqn:ML}
  \Tor{C_*(t)}{\cbasis[c']{*} \cup \cbasis[c'']{*}}{\emptyset} 
  = (-1)^\alpha \cdot 
  \Tor{C'_*(t)}{\cbasis[c']{*}}{\emptyset} \cdot
  \Tor{C''_*(t)}{\cbasis[c'']{*}}{\emptyset}
\end{equation}
where $\alpha \equiv \sum_j \alpha_{j-1}(C'_*(t)) \alpha_j(C''_*(t))
\mod 2$.  We first compute the sign in Equation~(\ref{eqn:ML}):
\begin{lemma}\label{lemma:sgn}
  The sign $(-1)^\alpha$ in Equation~(\ref{eqn:ML}) is given by $\alpha \equiv {b \cdot \dim C_3} \mod 2$.
\end{lemma}
\begin{proof}[Proof of Lemma~\ref{lemma:sgn}]
  It is easy to see from the definition that:
  \[
  \dim C_*''(t) = \dim C_*(t) - \dim C_*'(t).
  \]
  Moreover, $\dim C'_0(t) = 0 = \dim C'_3(t)$ and $\dim C'_1(t) = \dim
  C'_2(t) = 3b$, where $b$ is the number of boundary components of
  $M$.  Thus, reduced modulo 2, $\alpha_j(C_*'(t))$ are all zero
  except $\alpha_1(C_*'(t)) \equiv b \mod 2$.  As a consequence,
  \begin{align*}
    \alpha
    &\equiv \sum_j \alpha_{j-1}(C'_*(t)) \alpha_j(C''_*(t)) \\
    &\equiv \alpha_1(C'_*(t)) \alpha_2(C''_*(t)) \\
    &\equiv b \cdot (\dim C_0 + \dim C_1 + \dim C_2) \mod 2.
  \end{align*}
  Since the Euler characteristic of $M$ is equal to zero, we have that $\alpha \equiv b \cdot \dim C_3 \mod 2$.
\end{proof}

Next we compute the torsion of $C'_*(t)$ (with respect to the basis
$c'_*$).
\begin{lemma}\label{lemma:torsC'}
  We have:
  \begin{equation}\label{eqn:torsC'}
    \Tor{C'_*(t)}{\cbasis[c']{*}}{\emptyset} 
    = \prod_{\ell=1}^b (t^{a_\ell} - 1).
  \end{equation}
\end{lemma}
\begin{proof}[Proof of Lemma~\ref{lemma:torsC'}]
  It is easy to observe that (see Complex~(\ref{eqn:c'(t)})):
  \[
  \Tor{C'_*(t)}{\cbasis[c']{*}}{\emptyset}
  = \det \partial' = \prod_{\ell=1}^b (t^{a_\ell} - 1).
  \]
\end{proof}

If we substitute Equation~(\ref{eqn:torsC'}) into Equation~(\ref{eqn:ML}), 
we obtain:
\begin{equation}\label{eqn:D}
  \Tor{C''_*(t)}{\cbasis[c'']{*}}{\emptyset}
  = (-1)^{\alpha}\frac{
    \Tor{C_*(t)}{\cbasis[c']{*} \cup \cbasis[c'']{*}}{\emptyset}
  }{
    \prod_{\ell=1}^b (t^{a_\ell} - 1)
  }.
\end{equation}

Now we consider the limit of Equation~(\ref{eqn:D}) as $t$ goes to 1 and
prove the following lemma which gives a relation between the torsion
of $C''_*(t)$ and the non--abelian torsion of $M$ in the adjoint
representation.
\begin{lemma}\label{lemma:limtors}
  Let $\displaystyle{\delta_0 = {\tau_0} \cdot \prod_{i \mygeq 0} [\cbasis{i} /
    \cbasis[c']{i} \cup \cbasis[c'']{i}]^{(-1)^{i+1}} }$.  We have the following identity
  \begin{equation}\label{eqn:limtors}
    \lim_{t \to 1} \Tor{C''_*(t)}{\cbasis[c'']{*}}{\emptyset} = 
    (-1)^{b+\alpha} \cdot \delta_0 \cdot \RTors{M}{\smallbm{\lambda}}(\rho).
  \end{equation}
\end{lemma}

\begin{proof}[Proof of Lemma~\ref{lemma:limtors}]

  We begin the proof by some considerations on the complexes
  $C''_*(t)$ and $C_*(t)$ and their respective ``limits'' $C''_*$ and
  $C_*$ when $t$ goes to $1$.

  Return to the definition of the complex $C''_*(t) = C_*(t) /
  C'_*(t)$.  If $t$ goes to 1, then the acyclic complex $C'_*(t)$
  changes into the ``degenerated'' complex $C'_*$ in the sense that
  the map $\partial'$ becomes the zero map.  Hence $C'_1 \simeq
  H_1(C_*)$ and $C'_2 \simeq H_2(C_*)$.  Like $C''_*(t)$, the complex
  $C''_*$ is acyclic (see Lemma~\ref{lemma:acyclicity_cpxes}); more
  precisely if $t$ goes to 1, then we get in fact a complex related to
  the complex $C_*(M; \sllrho)$ but without homology, together with a
  distinguished basis different form the geometric one.  Repeat again
  that the twisted homology groups $H_*^\rho(M) = H_*(M; \sllrho)$ are
  endowed with the following distinguished bases (coming from the ones
  of $C'_*$, in fact it is not exactly a basis but a lift of the basis
  into $C_*$):
  \begin{enumerate}
  \item $H^\rho_1(M)$ is endowed with 
    $\hbasis{1} = \cbasis[c']{1} =
    \left\{ 
      \bbrack{P^\rho_1 \otimes \lambda_1}, 
        \ldots,
      \bbrack{P^\rho_b \otimes \lambda_b} 
    \right\}$,
  \item $H^\rho_2(M)$ is endowed with $\hbasis{2} = \cbasis[c']{2} =
    \left\{ 
       \bbrack{P^\rho_1 \otimes T^2_1}, 
         \ldots, 
       \bbrack{P^\rho_b \otimes T^2_b} 
    \right\}$.
  \end{enumerate}
  
  With obvious notation, choose a set of vectors $\bbasis[b'']{i+1}$ in
  $C''_{i+1}$ such that $\partial''_{i+1}(\bbasis[b'']{i+1})$ is a basis
  of $B''_i = \im(\partial''_{i+1}\colon C''_{i+1} \to C''_i)$.

  Observe that the set of vectors $1 \otimes \bbasis[b'']{i+1}$ in
  $C''_{i+1}(t)$ generates a subspace on which the boundary operator
  $\partial_{i+1}\colon C_{i+1}(t) \to C_i(t)$ is injective.
 
  With $\bbasis[b'']{3} = \cbasis[c'']{3}$ in mind, the torsion of $C''_*(t)$ (with respect to the basis
  $\cbasis[c'']{*}$) can be computed as follows:
  \begin{align*}
    \Tor{C''_*(t)}{\cbasis[c'']{*}}{\emptyset} &= \prod_{i=0}^{2}
    \left[
      \partial''_{i+1} (1 \otimes \bbasis[b'']{i+1})\, 1 \otimes
      \bbasis[b'']{i} / \cbasis[c'']{i}
    \right]^{(-1)^{i+1}} \\
    &= \prod_{i=0}^2 \left[ \cbasis[c']{i} \cup \partial_{i+1}
      (\bbasis[b]{i+1}) \bbasis[b]{i} / \cbasis[c']{i} \cup
      \cbasis[c'']{i} \right]^{(-1)^{i+1}}.
  \end{align*}
  Here $\bbasis[b]{i}$ denotes a lift of $1 \otimes \bbasis[b'']{i}$
  to $C_*(t)$.  As a result, we can rewrite 
  \begin{align}\label{eqn:torsC''}
    \Tor{C''_*(t)}{\cbasis[c'']{*}}{\emptyset}
    &= 
    \left[ 
      \cbasis[c']{2} \partial_3(\bbasis{3}) \bbasis{2} / \cbasis[c']{2} \cup \cbasis[c'']{2} 
    \right]^{-1} \\
   & \qquad
    \cdot
    \left[ 
      \cbasis[c']{1} \partial_2(\bbasis{2}) \bbasis{1} / 
      \cbasis[c']{1} \cup \cbasis[c'']{1} 
    \right]
    \cdot
    \left[
        \partial_1( \bbasis{1}) / \cbasis[c'']{0} 
      \right]^{-1}. \nonumber
  \end{align}
  We want now to relate Equation~(\ref{eqn:torsC''}) to an expression
  closer to the torsion of the twisted complex $C_*(M; \sllrho)$. For this we permute the
  vectors of $\partial_2(\bbasis{2})$ and $\partial_1(\bbasis{1})$
  with the ones of $\cbasis[c']{2}$ and $\cbasis[c']{1}$ in one of the
  determinants in Equation~(\ref{eqn:torsC''}). Each set of
  $\cbasis[c']{1}$ and $\cbasis[c']{2}$ consists of $b$ vectors and it
  is easy to observe that $\partial_3(\bbasis{3})$ consists of $\dim
  C_3(t) ( = \dim C_3)$ vectors and $\partial_2(\bbasis{2})$ consists
  of $\rk \partial_2 = \dim C_1(t) - b - \dim C_0(t) \, (\equiv \dim
  C_0 + \dim C_1 + b \mod 2)$ vectors.  Hence the sign arises from the permutation, 
   whose exponent is given by $b(\dim C_0 + \dim C_1 + b) + b
  \dim C_3$.  When we write $\varepsilon = (-1)^{b(\dim C_0 + \dim
    C_1)}$ and $(-1)^\alpha$ for $(-1)^{b \dim C_3}$ as in Lemma~\ref{lemma:sgn}, thus we have:
  \begin{align}\label{eqn:tors1}
    \Tor{C''_*(t)}{\cbasis[c'']{*}}{\emptyset}
    &= (-1)^{b + \alpha} \cdot \varepsilon \cdot
    \left[ 
      \partial_3(\bbasis{3}) \cbasis[c']{2} \bbasis{2} / \cbasis[c']{2} \cup \cbasis[c'']{2} 
    \right]^{-1} \\
    & \qquad \cdot
    \left[ 
        \partial_2(\bbasis{2}) \cbasis[c']{1} \bbasis{1} / 
        \cbasis[c']{1} \cup \cbasis[c'']{1} 
      \right]
      \cdot
      \left[ 
        \partial_1(\bbasis{1}) / \cbasis[c'']{0} 
      \right]^{-1}. \nonumber
        \end{align}
  By making a change of basis in Equation~(\ref{eqn:tors1}), we obtain the
  following expression:
  \begin{align}\label{eqn:tors2}
    \Tor{C''_*(t)}{\cbasis[c'']{*}}{\emptyset}
    &= (-1)^{b+\alpha} \cdot  
      (\, 
      \prod_{i \mygeq 0} \left[ \cbasis{i} / \cbasis[c']{i} \cup \cbasis[c'']{i} \right]^{(-1)^{i+1}} 
      \,)  \nonumber\\
    &\qquad \cdot \varepsilon
      \cdot \left[ \partial_3(\bbasis{3}) \cbasis[c']{2} \bbasis{2} / \cbasis{2} \right]^{-1} 
      \cdot \left[ \partial_2(\bbasis{2}) \cbasis[c']{1} \bbasis{1} / \cbasis{1} \right] 
      \cdot \left[ \partial_1(\bbasis{1}) / \cbasis{0} \right]^{-1}.
  \end{align}
  Moreover, using the definition of the bases $\cbasis[c']{1}$ and
  $\cbasis[c']{2}$, it is easy to observe that
  \begin{align}
    \lefteqn{ \lim_{t\to 1} \varepsilon \cdot
               \left[ \partial_3(\bbasis{3}) \cbasis[c']{2} \bbasis{2} / \cbasis{2} \right]^{-1} 
               \cdot 
               \left[ \partial_2(\bbasis{2}) \cbasis[c']{1} \bbasis{1} / \cbasis{1} \right] 
               \cdot \left[\partial_1(\bbasis{1}) / \cbasis{0} \right]^{-1}} 
    &  \nonumber \\
    &= \varepsilon \cdot
       \left[ \partial_3(\bbasis{3}) \hbasis[\tilde{h}]{2} \bbasis{2} / \cbasis{2} \right]^{-1} 
       \cdot 
       \left[ \partial_2(\bbasis{2}) \hbasis[\tilde{h}]{1} \bbasis{1} / \cbasis{1} \right] 
       \cdot 
       \left[ \partial_1(\bbasis{1}) / \cbasis{0} \right]^{-1} \label{eqn:tors3}\\
    &= \Tor{C_*}{\cbasis{*}}{\hbasis{*}}. \nonumber
  \end{align}
  The last step in Equations~$(\ref{eqn:tors3})$ is due to the fact that 
  \[
  (-1)^{|C_*|} 
    = (-1)^{\alpha_1(C_*)\beta_1(C_*)} 
    = (-1)^{(\dim C_1 + \dim C_0) b} 
    = \varepsilon.
  \]
  Hence, combining Equations~(\ref{eqn:tors2}) and~(\ref{eqn:tors3}), we
  obtain
  \[
  \lim_{t \to 1} \Tor{C''_*(t)}{\cbasis[c'']{*}}{\emptyset} 
  = (-1)^{b+\alpha}
  (\,
  \prod_{i \mygeq 0} 
    [ \cbasis{i}/\cbasis[c']{i} \cup \cbasis[c'']{i}]^{(-1)^{i+1}} 
  \,)
  \cdot \Tor{C_*}{\cbasis{*}}{\hbasis{*}}
  \]
  which is exactly the desired equality because
  $\RTors{M}{\smallbm{\lambda}}(\rho) = \tau_0 \cdot
  \Tor{C_*}{\cbasis{*}}{\hbasis{*}}$.
\end{proof}
  
We finish the proof by combining Equation~(\ref{eqn:D})
and~(\ref{eqn:limtors}) and using the definition of the polynomial
torsion $\PolyTors{M}{\otAd{\varphi}{\rho}}(t)$ given in
Equation~(\ref{eqn:DE}).

\subsection{Example: the non-abelian Reidemeister torsion of the Whitehead
  link exterior}\label{ex:Whitehead}

We apply Corollary~\ref{cor:DerFormula_one _variable} to the
polynomial torsion $\PolyTors{E_L}{\otAd{\varphi}{\rho}}(t_1, t_2)$ of
the Whitehead link exterior $E_L$ (see Fig.~\ref{fig:WhiteheadLinkMeridians}).  First we substitute $t$ into both
variables $t_1$ and $t_2$ in Equation~$(\ref{eqn:PolyTorsion_WL})$.  The
resulting homomorphism $\varphi_{(1,1)}$ corresponds to the induced
homomorphism $\pi_1(E_L) \to \pi_1(S^1)$ by the fibered structure of
$E_L$.  The dual surface is the fiber and is also a Seifert surface
for $L$.  Thus:
\[
\PolyTors{E_L}{\otAd{\varphi}{\rho}}(t,t) = \tau_0 \cdot
\frac{(t-1)^2}{t^3} \left( -2xyv t^2 + x^2 t(t+1)^2 + y^2 t (t+1)^2 -
  (t+1)^4 \right).
\]

Multiplying $(t-1)^{-2}$ and taking the limit for $t$ goes to $1$, we
obtain the hyperbolic torsion (the non--abelian Reidemeister torsion):
\begin{equation}\label{eqn:Whitehead}
  \RTors{E_L}{\smallbm{\lambda}}(\rho_{x, y, v}) 
  = \tau_0 \cdot 
  \left( 4(x^2+y^2)-16-2xyv \right).
\end{equation}
Since a lift of the holonomy representation is irreducible and the
traces of meridians are $\pm 2$, Equation~$(\ref{eqn:Whitehead})$ is also
valid at $x = \pm 2$.  Moreover the points $(\pm 2, \pm 2, 1 +
\sqrt{-1})$ and $(\pm 2, \pm 2, 1 - \sqrt{-1})$ correspond to lifts of
the holonomy representation and its complex conjugate.  When we
substitute $(\pm 2, \pm 2, 1 + \sqrt{-1})$ and $(\pm 2, \pm 2, 1 -
\sqrt{-1})$ into $(x, y, v)$, we have the following values of the
hyperbolic torsion (the non--abelian Reidemeister torsion) for the Whitehead link exterior and
its holonomy representation $\rho_0$ and the complex conjugate $\bar
\rho_0$
\[
\left\{\RTors{E_L}{\smallbm{\lambda}}(\rho_0),
\RTors{E_L}{\smallbm{\lambda}}(\bar \rho_0) \right\} = \left\{\tau_0 \cdot 8(1 \mp
\sqrt{-1})\right\}.
\]

\section{$\PolyTors{M}{\otAd{\varphi}{\rho}}$ is a polynomial}
\label{section:poly}

The torsion $\PolyTors{M}{\otAd{\varphi}{\rho}}$ is an element in the
fraction field $\CC(t_1, \ldots. t_n)$ by definition of the
Reidemeister torsion.  But actually, under a technical condition on
the representations $\rho \colon \pi_1(M) \to \SL$ and $\varphi \colon
\pi_1(M) \to \ZZ^n$, this torsion is in fact contained in $\CC[t_1,
\ldots, t_n]$, up to a factor $t_1^{m_1} \cdots t_n^{m_n}$ for some
integers $m_1, \ldots, m_n$.  This result is obtained by a cut and
paste argument by using the Multiplicativity Lemma for torsions.
This divisibility problem for link exteriors in $S^3$ has been also investigated by
Kitano and Morifuji~\cite{KitanoMorifuji05:divisibility} and Wada~\cite{Wada94}.

\subsection{Dual surfaces of $\varphi$ with rank one}\label{subsec:dual}

Using the universal coefficient theorem, a homomorphism $\varphi
\colon \pi_1(M) \to \ZZ$ can be regarded as a cohomology class in
$H^1(M, \ZZ)$, and its Poincar\'e dual $PD(\varphi)$ lies in
$H_2(M;\ZZ)$. Each representative of $PD(\varphi)$ consists of proper embedded
surfaces, \ie, embedded surfaces whose boundary is contained in $\bord
M$.  By Turaev~\cite{turaev:homol_estim}, we can choose proper
embedded surfaces satisfying that the complement $M \setminus S$ of $S$ in $M$ is connected
as a representative of $PD(\varphi)$.  We let $S_\varphi$ denote such 
representative surfaces.

In the case of the Reidemeister torsion with multivariable $(t_1,
\ldots, t_n)$, if we substitute $t_\ell = t^{a_\ell}$ for all $\ell$, then we have the homomorphism $\varphi_{(a_1, \ldots, a_n)}
\colon \pi_1(M) \to \ZZ$.  Moreover if $(a_1, \ldots,
a_n)$ are relatively prime positive integers then the composition $\varphi_{(a_1, \ldots, a_n)}$ satisfies assumption
$(A_\varphi)$.  When we regard $\varphi_{(a_1, \ldots, a_n)}$ as an
element in $H^1(M;\ZZ)$, we let $S_{(a_1, \ldots, a_n)}$ denote a
representative of $PD(\varphi_{(a_1, \ldots, a_n)})$ such that $M
\setminus S_{(a_1, \ldots, a_n)}$ is connected.

\subsection{A sufficient condition to be a polynomial}

Under some additional assumptions for the $\SL$-representation $\rho$
and dual surfaces determined by $\varphi$, we prove that the Reidemeister torsion
$\PolyTors{M}{\otAd{\varphi}{\rho}}$ is a polynomial.

\begin{theorem}\label{theorem:poly}

  Let $M$ be a hyperbolic three--manifold with
  tori boundary which satisfy $(A_M)$.  Let $\varphi$ be a surjective
  homomorphism $\pi_1(M) \to \ZZ^n$ which satisfies $(A_\varphi)$ and
  $\rho$ be an $\SL$-representation of $\pi_1(M)$ satisfying
  $(A_\rho)$.
  Let $S_\varphi = \cup_\ell S_\ell$ be the dual surfaces corresponding to
  $\varphi$.

  \begin{enumerate}
  \item Suppose that $n=1$. If the restriction $\rho|_{\pi_1(S_\ell)}$ of
    $\rho$ is non--abelian for all $\ell$, then $\PolyTors{M}{\otAd{\varphi}{\rho}}$
    is a polynomial in $t$, up to a factor $t^m$, $m \in \ZZ$.
  \item Suppose that $n \mygeq 2$. If for any natural number $N$ there
    exists relatively prime integers $(a_1, \ldots, a_n)$ such that
    $|a_i - a_j| > N$ for all distinct $i, j$ and the restriction
    $\rho|_{\pi_1(S)}$ on every component $S$ of the dual surfaces
    $S_{(a_1, \ldots, a_n)}$ corresponding to $\varphi_{(a_1, \ldots,
      a_n)}$ is non--abelian, then
    $\PolyTors{M}{\otAd{\varphi}{\rho}}$ is a polynomial in $t_1,
    \ldots, t_n$, up to a factor $t_1^{m_1} \cdots t_n^{m_n}$ for
    $m_1, \ldots, m_n \in \ZZ$.
  \end{enumerate}

\end{theorem}

Before proving this theorem, we give some explanations, examples and
counterexamples.

\begin{remark}
  Let $K$ be a hyperbolic knot in $S^3$. Consider $\rho_0 \colon
  \pi_1(E_K) \to \SL$ (a lift of) the holonomy representation and
  $\varphi \colon \pi_1(E_K) \to \ZZ$ the abelianization. The dual
  surface $S_\varphi$ corresponding to $\varphi$ is a Seifert surface of $K$. One can observe that the longitude 
  lies in the second commutator subgroup of $\pi_1(E_K)$. Thus,
  the restriction $\rho_0|_{\bord E_K}$ sends the longitude to a
  parabolic element which is not $\pm \I$.  This means that the
  restriction of $\rho_0$ on $\pi_1(S_\varphi)$ is non--abelian.
\end{remark}

\begin{remark}
  Let $K$ be a fibered knot in $S^3$. Consider an irreducible,
  non--metabelian representation $\rho \colon \pi_1(E_K) \to \SL$ and
  the homomorphism $\varphi \colon \pi_1(E_K) \to \pi_1(S^1) =\ZZ$
  induced by the fibration $E_K \to S^1$. The dual surface
  $S_\varphi$ corresponding to $\varphi$ is the fiber of $K$. One can easily
  observe that $\rho|_{\pi_1(S_\varphi)}$ is non--abelian, because the commutator
  subgroup $[\pi_1(E_K), \pi_1(E_K)]$ of $\pi_1(E_K)$ is $\pi_1(S_\varphi)$.
\end{remark}


\begin{remark}
For every $\SL$-representation $\rho$ and $\gamma \in \pi_1(M)$,
the linear map $Ad \circ \rho(\gamma)$ always has eigenvalue $1$.
In the one variable case with the ${\rm SL}_3(\CC)$-representation $Ad \circ \rho$, 
we can not use Wada's criterion~\cite[Proposition 8]{Wada94} directly.
\end{remark}

\subsection{Proof of Theorem~\ref{theorem:poly}}

The proof is divided into two main steps: we first consider the case
of a single variable $t$ and next use the naturality property of
$\PolyTors{M}{\otAd{\varphi}{\rho}}$ (see Section~\ref{section:naturality}) to deduce by an algebraic
argument the multivariable case form the one variable case.

\begin{remark}
  In what follows the sign in the Reidemeister torsion will be not
  relevant; so, we will work \emph{up to sign}.  We let
  $\Tor{W}{\cbasis{*}_W}{\emptyset}$ denote 
  the (acyclic) Reidemeister torsion of the manifold $W$ with 
  coefficients in $\sllt$ and computed in 
  the appropriate geometric basis $\cbasis{*}_W$.
\end{remark}

\subsubsection{Proof for one variable}

We cut the manifold $M$ along $S = S_\varphi$ and obtain the following splitting:
$M = N \cup (S \times I)$, where $I= [0,1]$ is the closed unit
interval. The boundaries of $N$ and $S \times I$ are equal and consist
in the disjoint union of two copies of $S$ denoted $S^- = S \times
\{0\}$ and $S^+ = S \times \{1\}$. We apply the Multiplicativity Lemma
to the Mayer--Vietoris sequence associated to this splitting to compute
the Reidemeister torsion of $M$. Our assumptions on $\rho \colon
\pi_1(M) \to \SL$ and $\varphi \colon \pi_1(M) \to \ZZ^n$ say that
every restrictions of $\rho$ to $\pi_1(S_\ell)$ are non--abelian,
so we have:
\begin{lemma}\label{lemma:twist-H}
  The twisted homology groups of $N$ and $S \times I$ are given by:
  \[
  H_*(S \times I; \sllrhot) = \CC(t) \otimes_\CC H_*(S \times
  I;\sllrho) \simeq \CC(t) \otimes H_*(S; \sllrho),
  \]
  \[
  H_*(N; \sllrhot) = \CC(t) \otimes_\CC H_*(N; \sllrho).
  \]
  Moreover, $H_0(S \times I; \sllrhot) = H_2(S \times I; \sllrhot) =
  0$.
\end{lemma}
\begin{proof}[Proof of Lemma~\ref{lemma:twist-H}]
  One has $M = N \cup (S \times I)$, where $N$ is a three--dimensional
  connected manifold whose boundary consists in $S^- \cup S^+$.  One
  can observe that the actions of the restrictions
  $\varphi_{|\pi_1(N)}$ and $\varphi_{|\pi_1(S_\ell)}$ of the
  representation $\varphi$ are trivial. So the first two equalities of the
  lemma hold.
  
  As $\rho|_{\pi_1(S_\ell)}$ is
  non--abelian, then $H_0(S_\ell \times I; \sllrhot) = 0$. In the case of a closed surface,
 last equality $H_2(S_\ell \times I; \sllrhot) = 0$ follows from Poincar\'e duality. For (compact) surface with boundary, last equality $H_2(S_\ell \times I; \sllrhot) = 0$ follows from the fact that $S_\ell \times I$ has the same homotopy type as a one--dimensional complex. So that $H_0(S \times I; \sllrhot) = H_2(S \times I; \sllrhot) = 0$.
\end{proof}

As $H_*(M; \sllrhot) = 0$, see Proposition~\ref{prop:acyclic}, the
Mayer--Vietoris sequence with coefficients in $\sllt$, denoted
$\mathcal{V}$, reduces to a single isomorphism:
\begin{equation}\label{eqn:MV}
  \mathcal{V}: 
  H_1(S^-;\sllrhot) \oplus H_1(S^+;\sllrhot) \xrightarrow{\simeq}
  H_1(N;\sllrhot) \oplus H_1(S \times I;\sllrhot) 
\end{equation}
where $H_1(S^\pm;\sllrhot) \simeq H_1(S;\sllrhot) \simeq H_1(S \times
I;\sllrhot)$.  The isomorphism in sequence~(\ref{eqn:MV}) is represented by the following matrix:
\[
\left(
  \begin{array}{cc}
    i^-_* & -i^+_* \\ 
    -\I & t\I
  \end{array}
\right),
\]
here $i^{\pm}\colon S_\pm \to N$ is the inclusion and $\I$ is the
identity matrix.  The Multiplicativity Lemma for Reidemeister
torsion gives us the identity below, because the common boundary of
$N$ and $S \times I$ is the disjoint union of two copies of $S$:
\[
\pm \PolyTors{M}{\rho}(t) \cdot \Tor{S}{\cbasis{*}_S}{\hbasis{*}_S}^2
\cdot \Tor{\mathcal{V}}{\{\hbasis{*}_S, \hbasis{*}_N\}}{\emptyset} =
\Tor{N}{\cbasis{*}_N}{\hbasis{*}_N} \cdot \Tor{S\times
  I}{\cbasis{*}_S}{\hbasis{*}_S}.
\]
Thus, since the torsions of $S$ and $S \times I$ are the same,
\[
\pm\PolyTors{M}{\rho}(t) = \Tor{\mathcal{V}}{\{\hbasis{*}_S,
  \hbasis{*}_N\}}{\emptyset}^{-1} \frac{
  \Tor{N}{\cbasis{*}_N}{\hbasis{*}_N} }{
  \Tor{S}{\cbasis{*}_S}{\hbasis{*}_S} } = \det \left(
  \begin{array}{cc}
    i^-_* & -i^+_* \\
    -\I & t\I
  \end{array}
\right) \frac{ \Tor{N}{\cbasis{*}_N}{\hbasis{*}_N} }{
  \Tor{S}{\cbasis{*}_S}{\hbasis{*}_S} }.
\]
The fraction of torsions $\Tor{N}{\cbasis{*}_N}{\hbasis{*}_N} /
\Tor{S}{\cbasis{*}_S}{\hbasis{*}_S}$ is independent of $t$ (because as
we have already observed in the proof of Lemma~\ref{lemma:twist-H},
the actions of the restrictions $\varphi_{|\pi_1(N)}$ and
$\varphi_{|\pi_1(S)}$ are trivial).  This proves that, up to sign,
$$
\pm \PolyTors{M}{\rho}(t) = \det(t i^-_* - i^+_*) \frac{
  \Tor{N}{\cbasis{*}_N}{\hbasis{*}_N} }{
  \Tor{S}{\cbasis{*}_S}{\hbasis{*}_S} }
$$ is a polynomial in $t$.
The one variable case in Theorem~\ref{theorem:poly} is proved.

\begin{remark}
  A result similar to Lemma~\ref{lemma:twist-H} can be found
  in~\cite[Proposition 3.6]{KL}.
\end{remark}

\subsubsection{Proof from one variable to two variables}

Suppose that $\varphi\colon \pi_1(M) \to \ZZ \oplus \ZZ$. A priori, up to
multiplications by $t_1^{r}t_2^{s}$, 
the torsion $\PolyTors{M}{\otAd{\varphi}{\rho}}(t_1, t_2)$ is a
rational function  $P(t_1, t_2) / Q(t_1, t_2)$, where $P(t_1, t_2)$ and
$Q(t_1, t_2) \ne 0$ are coprime in $\CC[t_1, t_2]$. We will prove in fact, reducing
the situation to one variable, that the polynomial $Q(t_1,t_2)$ is
constant.

To this end, suppose that $Q(t_1, t_2)$ is a non--constant
polynomial. Without loss of generality we assume that $Q(t_1, t_2)$ is
a non--constant in $t_2$.  Applying Euclidean algorithm to $P(t_1,
t_2)$ and $Q(t_1, t_2)$ in $\CC(t_1)[t_2]$, we obtain the following
equality in the polynomial ring $\CC(t_1)[t_2]$ over the rational
function field $\CC(t_1)$:
\[
P(t_1, t_2)\tilde u(t_1, t_2) + Q(t_1, t_2) \tilde v(t_1, t_2) = 1.
\]
Here the coefficients of the polynomials in $t_2$ $\tilde u(t_1, t_2)$ and $\tilde v(t_1, t_2)$ are rational
functions in $\CC(t_1)$.  By taking product with some polynomial
$w(t_1) \in \CC[t_1]$, the following equality holds in $\CC[t_1, t_2]$:
\begin{equation}
  \label{eqn:EuclidAlgo_two_variables}
  P(t_1, t_2) u(t_1, t_2) + Q(t_1, t_2) v(t_1, t_2) = w(t_1).
\end{equation}
To each  pair of coprime integers $(a_1, a_2)$, consider the
homomorphism $h_{(a_1, a_2)}\colon \ZZ^2 \to \ZZ$ defined by $h_{(a_1,
  a_2)} (t_1) = t^{a_1}$ and $h_{(a_1, a_2)} (t_2) = t^{a_2}$ and let
$\varphi_{(a_1, a_2)} = h_{(a_1, a_2)} \circ \varphi$.  Using
Proposition~\ref{cor:naturality} and by the first part of the proof (for
one variable), we know that:
\[
\PolyTors{M}{\otAd{\varphi_{(a_1, a_2)}}{\rho}}(t) =
\PolyTors{M}{\otAd{\varphi}{\rho}}(t^{a_1}, t^{a_2}) = \frac{P(t^{a_1}, t^{a_2})}{Q(t^{a_1}, t^{a_2})} \in \CC[t].
\]
Thus there exists a polynomial $R(t) \in \CC[t]$ such that $P(t^{a_1}, t^{a_2})= R(t) Q(t^{a_1},
t^{a_2})$.  Hence by substituting $t^{a_1}$ and $t^{a_2}$ to $t_1$
and $t_2$ respectively, Equation~$(\ref{eqn:EuclidAlgo_two_variables})$ turns into
\[
Q(t^{a_1}, t^{a_2}) \{ R(t) u(t^{a_1}, t^{a_2}) + v(t^{a_1},
t^{a_2})\} = w(t^{a_1}).
\]

From our assumption $a_2$ can be chosen sufficiently large.  Since we
suppose that $Q(t_1, t_2)$ is not constant in $t_2$, by changing $a_2$
into a sufficient large integer, we obtain an arbitrary large degree
polynomial $Q(t^{a_1}, t^{a_2})$.  This contradicts the fact that the
degree of $Q(t^{a_1}, t^{a_2})$ must be less or equal to that of
$w(t^{a_1})$.  Therefore $Q(t_1, t_2)$ is constant in $t_2$.

Similarly, we can also
conclude that $Q(t_1, t_2)$ is constant in $t_1$.  As a consequence, the
polynomial $Q(t_1, t_2)$ is constant which proves that
$\PolyTors{M}{\otAd{\varphi}{\rho}}(t_1, t_2)$ is a polynomial in two
variables.

By an inductive argument (in descending order of $a_i$), the general
case works in the same way: $\PolyTors{M}{\otAd{\varphi}{\rho}}(t_1,
\ldots, t_n) \in \CC[t_1, \ldots, t_n]$. The multivariable case in
Theorem~\ref{theorem:poly} is also proved.


\begin{remark}
  Actually, in the multivariable case of Theorem~\ref{theorem:poly},
  it is sufficient to assume that there exists some positive integer
  $N_0$ such that for any $N \mygeq N_0$ we have a dual surfaces
  $S_{(a_1, \ldots, a_n)}$ satisfying that $|a_i -a_j| > N$ for all
  distinct $i, j$.
\end{remark}

\section{Reciprocality of the polynomial torsion (with sign)}
\label{section:symmetries}

\subsection{Reciprocality properties}

The Alexander polynomial $\Delta_K(t)$ of a knot $K \subset S^3$ is
known to be reciprocal in the sense that $\Delta_K(t^{-1}) = \pm t^s
\Delta_K(t)$ from a long time. This property was first observed by
Seifert~\cite{Seifert:1934} and is a consequence of
Poincar\'e duality. Milnor~\cite{Milnor:1962} proved
it again using the interpretation of the Alexander polynomial as an
abelian Reidemeister torsion. Kitano~\cite{Kitano}, Kirk and
Livingston~\cite{KL} observe that Milnor's argument work fine in the
context of twisted Alexander polynomial. Hillman, Silver and
Williams~\cite{HSW:09} give a more general discussion on reciprocality
of twisted Alexander invariants for representations into
$\mathrm{SL}_n(\CC)$. All the known reciprocality formulas are
sign--less. In our situation Milnor's argument also work, and using
the fact that the torsion of $\partial M = \bigcup_{\ell = 1}^b
T_\ell^2$ is trivial, we have the following result, in which the sign
is analyzed in details (\cf~\cite[Theorem 2]{Milnor:1962} and
\cite[Theorem 5.1]{KL}).
We also refer to \cite{FriedlKimKitayama_Poincare}
for duality formulas for the polynomial torsion in more general situations.

\begin{theorem}\label{theorem:symmetry_torsion}
  Let $\bar{\cdot} \colon \ZZ \to \ZZ$ be the involution defined by
  $\bar{t} = t^{-1}$.  If we consider a homomorphism 
  $\varphi \colon \pi_1(M) \to \ZZ$ and its composition $\bar{\varphi}$ 
  with the involution $\bar{\cdot}$, 
  then the polynomial torsion satisfies the following identity 
  up to a factor $t^k\, (k \in \ZZ)$:
  \[
  \PolyTors{M}{\otAd{\varphi}{\rho}}(t^{-1}) = 
  \PolyTors{M}{\otAd{\bar{\varphi}}{\rho}}(t) = 
  \epsilon (-1)^{b(b+1)/2} \PolyTors{M}{\otAd{\varphi}{\rho}}(t).
  \]
  Here $\epsilon$ is the sign of the Reidemeister torsion of the long
  exact sequence in homology associated to the pair $(M, \bord M)$
  with the basis given by the bases of $H_*(\bord M;\IR)$, $H_*(M;\IR)$
  and the Poincar\'e dual bases of $H_*(M;\IR)$ in $H_*(M, \bord
  M;\IR)$ as in Equation~$(\ref{eqn:bases_the_pair})$.
\end{theorem}

\begin{remark}[A duality property for link exteriors]
  As a special case, if $M$ is a link exterior $E_L =S^3 \setminus N(L)$,
  then we can explicitly compute the sign--term in
  Theorem~\ref{theorem:symmetry_torsion}. 
  Proposition~\ref{prop:sign_pair_link_exterior} in Appendix
  gives us the sign $\epsilon = (-1)^{b(b-1)/2}$, where $b$ is the number of boundary components of $M$. Using naturality of
  the polynomial torsion (see Section~\ref{section:naturality}) then the
  following duality for the polynomial torsion
  $\PolyTors{E_L}{\otAd{\varphi}{\rho}}$ holds:
  \begin{align}
    \PolyTors{E_L}{\otAd{\varphi}{\rho}}(t^{-1}) 
      &=(-1)^{b(b-1)/2} \cdot (-1)^{b(b+1)/2} \PolyTors{E_L}{\otAd{\varphi}{\rho}}(t) \nonumber\\
      &=(-1)^b \PolyTors{E_L}{\otAd{\varphi}{\rho}}(t) \label{eqn:symmLink}.
  \end{align}
  In particular, if $L$ is a knot $K$ in $S^3$, then the
  following duality holds:
  \begin{equation}\label{eqn:symmKnot}
    \PolyTors{E_K}{\otAd{\varphi}{\rho}}(t^{-1}) =  
     - \PolyTors{E_K}{\otAd{\varphi}{\rho}}(t).
  \end{equation}
  Equation~(\ref{eqn:symmKnot}), and more generally
  Equation~(\ref{eqn:symmLink}), can be considered as a sign--refined
  version of Milnor's duality Theorem~\cite{Milnor:1962} for
  Reidemeister torsion.
\end{remark}

Our attention is restricted to the composition of an
$\SL$-rerpesentation and the adjoint action.  We refer
to~\cite{FriedlKim:06, FriedlVidussi:survey,
  HSW:09} for reciprocality formulas for other types
of representations.

\subsection{Proof of Theorem~\ref{theorem:symmetry_torsion}}

The proof is essentially based on the Multiplicativity Lemma for
torsions (with sign).  We apply the Multiplicativity Lemma in
Section~\ref{subsection:multiplicativitylemma} for short exact
sequences with the coefficient $\IR$ and $\sllt$ to observe
the relation between the signed torsion of the original chain complex
and that of the dual chain complex.  
With the fact that
$\alpha(C', C'')$ for the coefficients in $\IR$ are same as $\alpha(C',
C'')$ for the coefficients in $\sllt$ in mind, 
the Multiplicativity Lemma for the pair $(M, \bord M)$ yields the following equation for 
the sign--refined torsions:
\begin{equation}\label{eqn:1st_step_duality}
  \PolyTors{M}{\otAd{\varphi}{\rho}}(t)
  =\epsilon(-1)^{\nu}\cdot
  \PolyTors{\bord M}{\otAd{\varphi}{\rho}}(t) \cdot
  \PolyTors{(M, \bord M)}{\otAd{\varphi}{\rho}}(t).
\end{equation}
Here $\nu \in \ZZ/2\ZZ$ is the sign given by the sum
$\nu = 
\sum_{i=0}^3 (\beta_i+1)(\beta'_i+\beta^{''}_i)+\beta'_{i-1}\beta^{''}_i 
\in \ZZ/2\ZZ$ where
$$\beta_i = 
\sum_{r=0}^i \dim_{\IR}H_r(M;\IR), \;\beta'_i = \sum_{r=0}^i
\dim_{\IR}H_r(\bord M;\IR) \text{ and } \beta^{''}_i = \sum_{r=0}^i
\dim_{\IR}H_r(M, \bord M;\IR).
$$

We compute each terms appearing in Equation~(\ref{eqn:1st_step_duality}).

\subsubsection{Computation of the sign $\nu$.}

By our assumption on $M$ and $\bord M$, we observe that $\nu = 1$
in $\ZZ/2\ZZ$.

\subsubsection{Computation of $\PolyTors{\bord M}{\otAd{\varphi}{\rho}}(t)$.}

By using the Multiplicativity Lemma, we have 
\[
\PolyTors{\cup_{\ell=1}^b T^2_\ell}{\otAd{\varphi}{\rho}}(t) 
= (-1)^{b-1}
\PolyTors{T^2_b}{\otAd{\varphi}{\rho}}(t)
\cdot
\PolyTors{\cup_{\ell=1}^{b-1} T^2_\ell}{\otAd{\varphi}{\rho}}(t).
\]
Hence we have
$$
  \PolyTors{\bord M}{\otAd{\varphi}{\rho}}(t) = 
  (-1)^{b(b-1)/2}
  \prod_{\ell=1}^b
  \PolyTors{T^2_\ell}{\otAd{\varphi}{\rho}}(t).
$$
Furthermore a direct computation provides the following lemma:
\begin{lemma}\label{lemma:Torus}
  The torsion $\PolyTors{T^2_\ell}{\otAd{\varphi}{\rho}}(t)$ of the
  $\ell$-th component of $\bord M$ with the homology orientation given by
  the ordered basis $\{\bbrack{T^2_\ell}, \bbrack{\lambda_\ell},
  \bbrack{\mu_\ell},\bbrack{p_\ell}\}$ is equal to $+1$.
\end{lemma}

As a consequence, the twisted Alexander invariant of $\partial M$ is given by:
\begin{equation}\label{Eq:torsionbound}
\PolyTors{\bord M}{\otAd{\varphi}{\rho}}(t) = (-1)^{b(b-1)/2}.
\end{equation}

\subsubsection{Computation of $\PolyTors{(M, \bord M)}{\otAd{\varphi}{\rho}}(t)$.}
Here we use the duality properties of torsion proved by M. Farber and
V. Turaev in ~\cite{FarberTuraev:PR_metric}. We write the torsion
$\PolyTors{(M, \bord M)}{\otAd{\varphi}{\rho}}(t)$ in the right hand
side of Equation~$(\ref{eqn:1st_step_duality})$ as the torsion in the left
hand side with a sign term. Let $(M', \bord M')$ denotes the dual cell
decomposition of $(M, \bord M)$. Using the invariance of the Reidemeister torsion under
subdivisions of CW--pairs we obtain :
\[
\PolyTors{(M, \bord M)}{\otAd{\varphi}{\rho}}(t) = \PolyTors{(M',
  \bord M')}{\otAd{\varphi}{\rho}}(t).
\]
Let ${\rho}^{\star}\colon \pi_1(M) \to \SL$ be the following
representation
\[
{\rho}^{\star} = \left(
  \begin{smallmatrix}
    0 & 1\\
    -1& 0
  \end{smallmatrix}
\right)\rho \left(
  \begin{smallmatrix}
    0 & 1\\
    -1& 0
  \end{smallmatrix}
\right)^{-1}.
\]
Using the invariance of the Reidemeister torsion under conjugation of representations
we thus have
\[
  \PolyTors{(M, \bord M)}{\otAd{\varphi}{\rho}}(t)
  = \PolyTors{(M', \bord M')}{\otAd{\varphi}{\rho}}(t) 
  = \PolyTors{(M', \bord
    M')}{\otAd{\varphi}{{\rho}^{\star}}}(t).
\]

Observe that the pair $(M', \bord M')$ and the representation
$\otAd{\varphi}{{\rho}^{\star}}$ give the dual chain complex
of $M$ twisted by the representation $\otAd{\overline{\varphi}}{\rho}$, \ie, the
$\sllt_{{\rho}^{\star}}$-twisted chain complex $C_*(M', \bord
M';\sllt_{{\rho}^{\star}})$ can be identified with the dual chain complex of
$C_*(M;\msllrho{\overline{t}})$.  By using this identification and the
duality of torsion in~\cite{FarberTuraev:PR_metric}, the torsion
$\PolyTors{(M, \bord M)}{\otAd{\varphi}{\rho}}(t)$ is expressed as
\begin{align}\label{eqn:dualtorsion}
  \PolyTors{(M, \bord M)}{\otAd{\varphi}{\rho}}(t) &=
  \PolyTors{(M', \bord M')}{\otAd{\varphi}{{\rho}^{\star}}}(t) \nonumber \\
  &=
  (-1)^{s(H_*(M;\IR))+s(C_*(M;\IR))}\cdot(-1)^{s(C_*(M;\sllrhot))}\cdot
  \PolyTors{M}{\otAd{\overline{\varphi}}{\rho}}(t).
\end{align}
Here $s(C_*) = \sum_{q=1}^m \alpha_{q-1}(C_*)\alpha_q(C_*)$ for a chain
complex $C_*=C_m\oplus \cdots \oplus C_0$, with
$\alpha_q(C_*)=\sum_{j=0}^q\dim C_j$.  Further observe the sign term
difference between Farber--Turaev's formula and
Equation~(\ref{eqn:dualtorsion}): here the sign term
$\sum_{q=0}^{(m-1)/2}\alpha_{2q}(C_*)$ used
in~\cite{FarberTuraev:PR_metric} is omitted from $s(C_*)$ because we use
the complex $C_*(M', \bord M';\sllrhot)$ instead of the dual complex
of $C_*(M;\sllrhot)$.

The first sign $(-1)^{s(H_*(M;\IR))+s(C_*(M;\IR))}$
in~Equation~(\ref{eqn:dualtorsion}) comes from the duality of the sign
terms $\tau_0$ in $\PolyTors{(M, \bord M)}{\otAd{\varphi}{\rho}}(t)$
and the second one $(-1)^{s(C_*(M;\sllrhot))}$
in~Equation~(\ref{eqn:dualtorsion}) comes from the duality of torsions for
$C_*(M; \sllrhot)$.  Since $s(C_*(M;\IR))$ and $s(C_*(M;\sllrhot))$
are equal in $\ZZ/2\ZZ$, it turns out that
\begin{equation}\label{eqn:duality_torsion}
  \PolyTors{(M, \bord M)}{\otAd{\varphi}{\rho}}(t)
  =
  (-1)^{s(H_*(M;\IR))}
  \PolyTors{M}{\otAd{\overline{\varphi}}{\rho}}(t).
\end{equation}

\subsubsection{Conclusion}
Substituting $\nu =1$, $s(H_*(M;\IR)) = b+1$ and
Equation~(\ref{eqn:duality_torsion}) into
Equation~$(\ref{eqn:1st_step_duality})$ we obtain that
\begin{equation}\label{eqn:2nd_step_duality}
  \PolyTors{M}{\otAd{\varphi}{\rho}}(t)
  =\epsilon (-1)^{b}
  \PolyTors{\bord M}{\otAd{\varphi}{\rho}}(t)
  \PolyTors{M}{\otAd{\overline{\varphi}}{\rho}}(t).
\end{equation}
Finally, using
Equation~(\ref{Eq:torsionbound}),
we can see that Equation~(\ref{eqn:2nd_step_duality}) turns into
$$
\PolyTors{M}{\otAd{\varphi}{\rho}}(t)
=\epsilon (-1)^{b+b(b-1)/2}
\PolyTors{M}{\otAd{\overline{\varphi}}{\rho}}(t)
=\epsilon (-1)^{b(b+1)/2}
\PolyTors{M}{\otAd{\overline{\varphi}}{\rho}}(t)
$$
which achieves the proof of Theorem~\ref{theorem:symmetry_torsion}.
\qed

\appendix

\section{A natural homology orientation for  link exterior}
\label{section:link_orientation}

In this section, we exhibit a natural, and in a sense compatible,
homology orientation in the case where $M$ is the exterior $E_L = S^3
\setminus N(L)$ of a link $L$ in $S^3$. Here $N(L)$ denotes a tubular
neighborhood of $L$. Observe that the boundary $\bord M$ of $M$
consists in the disjoint union of $b$ tori $T^2_1, \ldots, T^2_b$. For more details on homology orientations the reader is invited to refer to Turaev's monograph~\cite{Turaev:2002}.
	
The definition of the homology orientation needs some orientation
conventions. We suppose that $S^3$ and $L$ are oriented. 
Thus $E_L = S^3 \setminus N(L)$ and each torus $T^2_\ell$ inherits the orientation
induced by the one of $S^3$. Moreover each torus $T^2_\ell$ is given
together with its peripheral-system $(\lambda_\ell, \mu_\ell)$, where
$\lambda_\ell$, the longitude, and $\mu_\ell$, the meridian, generates
$H_1(T^2_\ell; \ZZ)\simeq \ZZ \oplus \ZZ$. These two curves are oriented using the following
rules: $\mu_\ell$ is oriented using the convention 
$\lk (\mu_\ell, L) = +1$, and $\lambda_\ell$ is oriented using the convention
$\mathrm{int}(\mu_\ell, \lambda_\ell) = +1$.

In what follows, we construct natural homology orientations for the
link exterior $E_L$, for its boundary $\bord E_L = \bigcup_{\ell = 1}^b T^2_\ell $ and 
for the pair $(E_L, \bord E_L)$. 
	
\subsection{Homology groups and homology orientations}
\label{A:1}

We begin our investigations by describing in details the homology
groups of the link exterior $E_L$, of its boundary 
$\bord E_L = \bigcup_{\ell = 1}^b T^2_\ell $ and of the pair $(E_L, \bord E_L)$.

The homology groups of the torus $T^2_\ell$ are given as follows:
$$
H_2(T^2_\ell;\IR) = \IR \, \bbrack{T^2_\ell},\,
H_1(T^2_\ell;\IR) = \IR \, \bbrack{\lambda_\ell} \oplus \IR \, \bbrack{\mu_\ell},\,
H_0(T^2_\ell;\IR) = \IR \, \bbrack{p_\ell}.
$$
Here $\bbrack{T^2_\ell}$ is the fundamental class induced by the orientation of
$T^2_\ell$, $\bbrack{\mu_\ell}$ and $\bbrack{\lambda_\ell}$ are given by the oriented
meridian and the oriented longitude (see above) and $\bbrack{p_\ell}$ is the
homology class of the base point.

An application of Mayer-Vietoris sequence associated to the following decomposition of the three--sphere:
$S^3 = E_L \cup (\bigcup_{\ell = 1}^b S^1 \times D^2_\ell)$ gives us the following
generators for the homology groups of $E_L$:
$$
H_2(E_L;\IR) = \bigoplus_{\ell=1}^{b-1} \IR \, \bbrack{T^2_\ell},\,
H_1(E_L;\IR) = \bigoplus_{\ell=1}^b \IR \, \bbrack{\mu_\ell},\,
H_0(E_L;\IR) = \IR \, \bbrack{p_b}.
$$

To describe the homology groups of the pair $(E_L, \bord E_L)$ we use
Poincar\'e duality and get the following generators:
\begin{align*}
H_3(E_L, \bord E_L;\IR) &= \IR \, \bbrack{E_L, \bord E_L},\\
H_2(E_L, \bord E_L;\IR) &= \IR \, \bbrack{S_1} \oplus \cdots \oplus \IR \, \bbrack{S_b},\\
H_1(E_L, \bord E_L;\IR) &= \IR \, \bbrack{\gamma_1} \oplus \cdots \oplus \IR \, \bbrack{\gamma_{b-1}}.
\end{align*}
Here $\bbrack{E_L, \bord E_L} = \bbrack{p_b}^*$ is the fundamental class induced by
orientations. If $S_\ell$ denotes the restriction of a Seifert surface of
the component $K_\ell$ in $L$ to $E_L$, then its class $\bbrack{S_\ell} = \bbrack{\mu_\ell}^*$. 
And finally, if $\gamma_\ell$ denotes a path connecting the
point $p_\ell$ to $p_b$, then its class $\bbrack{\gamma_\ell} = \bbrack{T_\ell^2}^*$.

The \emph{homology orientations} we fix on $\bord E_L$, $E_L$ and 
$(E_L, \bord E_L)$ are respectively induced by the following ordered bases:
\begin{align}\label{eqn:bases_the_pair}
  H_*(\bord E_L;\IR)
  &= \left\langle { 
       \bbrack{T^2_1}, \ldots, \bbrack{T^2_b}, 
       \bbrack{\lambda_1}, \bbrack{\mu_1}, 
       \ldots, 
       \bbrack{\lambda_b}, \bbrack{\mu_b}, 
       \bbrack{p_1}, \ldots, \bbrack{p_b} 
     } \right\rangle_\IR, \nonumber \\
  H_*(E_L;\IR)
  &= \left\langle { 
       \bbrack{T^2_1}, \ldots, \bbrack{T^2_{b-1}}, 
       \bbrack{\mu_1}, \ldots, \bbrack{\mu_b}, 
       \bbrack{p_b}
     } \right\rangle_\IR, \\
  H_*(E_L, \bord E_L;\IR) 
  &= \left\langle { 
       \bbrack{E_L, \bord E_L}, 
       \bbrack{S_1}, \ldots, \bbrack{S_b}, 
       \bbrack{\gamma_1}, \ldots, \bbrack{\gamma_{b-1}}
     }\right\rangle_\IR. \nonumber
\end{align}

\subsection{A sign term}

This combination of bases for the pair $(E_L, \bord E_L)$ are natural in
the sense that they are given by Poincar\'e duality. In the case of
link exteriors, using such homology orientations we will fix numbers
of sign indeterminacy in our formulas. However, to be exhaustive it
remains to us to explicitly compute the Reidemeister torsion of the
long exact sequence in homology associated to the pair $(E_L, \bord E_L)$. 
This torsion is a new sign term, given in the following
proposition, which will give us the sign--term in the symmetry
formula for the polynomial torsion (see
Section~\ref{section:symmetries}, in particular
Theorem~\ref{theorem:symmetry_torsion} and Equation~(\ref{eqn:symmLink})).

\begin{proposition}\label{prop:sign_pair_link_exterior}
  The torsion of the long exact sequence in homology associated to the
  pair $(E_L, \bord E_L)$, in which homology groups are endowed with the
  distinguished bases given in Equation~$(\ref{eqn:bases_the_pair})$, is
  equal to $(-1)^{b(b-1)/2}$.
\end{proposition}

\begin{proof}
  Let $\mathcal{H}_*$ denotes the long exact sequence in homology
  associated to the pair $(E_L, \bord E_L)$.  Counting dimensions, it is
  easy to observe that $\mathcal{H}_*$ is decomposed into the 
  following three short exact sequences:
  \begin{equation}\label{eqn:short_exact_pari}
    \mathcal{H}^{(i)}_* : 
        0 \to 
        H_{i+1}(E_L, \bord E_L;\IR) \xrightarrow{\delta^{(i)}} 
        H_i(\bord E_L;\IR) \xrightarrow{-j^{(i)}_*}
        H_i(E_L;\IR) \to
        0 
  \end{equation}
  for $i=0, 1, 2$. Here $\delta^{(i)}$ denotes the connecting
  homomorphism and $j^{(i)}_*$ is induced by the usual inclusion
  $\bord E_L \hookrightarrow E_L$.  Thus, the torsion of $\mathcal{H}_*$
  is equal to the alternative product of the three torsions of the
  above short exact sequences~(\ref{eqn:short_exact_pari}).

  The torsions are computed with respect to the bases given in
  Equation~$(\ref{eqn:bases_the_pair})$ and we have:

  \begin{claim}
    The Reidemeister torsions of the short exact sequences $\mathcal{H}^{(i)}_*$ are given by:
    \begin{align*}
    \Tor{\mathcal{H}^{(2)}_*}{\{\hbasisrel{E_L}\}}{\emptyset}
      &= (-1)^{b-1},\\
      \Tor{\mathcal{H}^{(1)}_*}{\{\hbasisrel{E_L}\}}{\emptyset}
      &= (-1)^{b(b-1)/2},\\
      \Tor{\mathcal{H}^{(0)}_*}{\{\hbasisrel{E_L}\}}{\emptyset} 
      &= (-1)^{b-1}.
    \end{align*}
  \end{claim}
  \begin{proof}[Proof of the claim]
    Each torsion of the short exact sequences $\mathcal{H}^{(i)}_*$
    can be calculated as follows.
    \begin{itemize}
    \item \emph{Computation of 
    $\Tor{\mathcal{H}^{(2)}_*}{\{\hbasisrel{E_L}\}}{\emptyset}$}.

  The connecting homomorphism $\delta^{(2)}$ maps $\bbrack{E_L, \bord
    E_L}$ to $\bbrack{T^2_1} + \cdots + \bbrack{T^2_b}$.  Moreover one
  has 
  $$j_*^{(2)}(\bbrack{T^2_{\ell}}) = \bbrack{T^2_{\ell}} \text{ for } \ell
  = 1, \ldots, b-1.$$
    Thus the set $\{\bbrack{T^2_1}, \ldots,
  \bbrack{T^2_{b-1}}\}$ of vectors in $H_2(\bord E_L; \IR)$ can be
  chosen as lifts of the distinguished basis of $H_2(E_L; \IR)$.  As a
  result, the torsion of this short exact sequence
  $\mathcal{H}_*^{(2)}$ is given by the following base change
  determinant:
    \begin{align*}
      \lefteqn{ 
        \Tor{\mathcal{H}^{(2)}_*}{\{\hbasisrel{E_L}\}}{\emptyset}
      } &\\
      &= \left[ 
        \{ \delta(
           \bbrack{E_L, \bord E_L}), 
           \bbrack{T^2_1}, \ldots, \bbrack{T^2_{b-1}}
        \} / 
        \{\bbrack{T^2_1}, \bbrack{T^2_2}, \ldots, \bbrack{T^2_b} \}
        \right] \\
      &= \left|
        \begin{array}{cccc}
          1 & 1 &  & \\
          \vdots & & \ddots & \\
          1 & & & 1 \\
          1 &   &  & 
        \end{array}
        \right| =(-1)^{b-1}.
    \end{align*}

  \item \emph{Computation of
      $\Tor{\mathcal{H}^{(1)}_*}{\{\hbasisrel{E_L}\}}{\emptyset} $}.

    The connecting homomorphism $\delta^{(1)}$ maps $\bbrack{S_\ell}$
    to $\bbrack{\lambda_\ell} + \sum_{k \ne \ell} \ell k(K_\ell,
    K_k)\bbrack{\mu_k}$ where $\lk$ denotes the linking number.
    Moreover, $j_*^{(1)}(\bbrack{\mu_\ell}) = \bbrack{\mu_\ell}$, for
    $\ell=1, \ldots, b$.  Thus, the set $\{\bbrack{\mu_1}, \ldots,
    \bbrack{\mu_b}\}$ of vectors in $H_1(\bord E_L; \IR)$ can be
    chosen as lifts of the distinguished basis of $H_1(E_L; \IR)$.  It
    follows that the torsion of the short exact sequence
    $\mathcal{H}_*^{(1)}$ is given by the ratio of the following two determinants of 
    bases change matrices:
      \begin{align}
        \lefteqn{
          \Tor{\mathcal{H}^{(1)}_*}{\{\hbasisrel{E_L}\}}{\emptyset}
        } & \nonumber \\
        &= 
          \left[ 
          \{ \delta(\bbrack{S_1}), \ldots, \delta(\bbrack{S_b}), 
             \bbrack{\mu_1}, \ldots, \bbrack{\mu_b} 
          \} / \{ 
              \bbrack{\lambda_1}, \bbrack{\mu_1}, 
              \ldots, 
              \bbrack{\lambda_b}, \bbrack{\mu_b}
          \}
          \right]. \label{eqn:torsion_2nd_short_exact}
      \end{align}
      Observe that the determinant in the right-hand side of
      Equation~(\ref{eqn:torsion_2nd_short_exact}) can be written again as
      the product of the following two base change determinants:
      $$
      D_1 = 
        \left[ 
        \{ 
          \delta(\bbrack{S_1}), \ldots, \delta(\bbrack{S_b}), 
          \bbrack{\mu_1}, \ldots, \bbrack{\mu_b}
        \} / \{
          \bbrack{\lambda_b}, \ldots, \bbrack{\lambda_1},
          \bbrack{\mu_1}, \ldots, \bbrack{\mu_b}
        \} 
        \right],
      $$
      $$ 
      D_2 = 
        \left[ 
        \{ 
           \bbrack{\lambda_b}, \ldots, \bbrack{\lambda_1}, 
           \bbrack{\mu_1}, \ldots, \bbrack{\mu_b}
        \} / \{ 
           \bbrack{\lambda_1}, \bbrack{\mu_1}, 
           \ldots, 
           \bbrack{\lambda_b}, \bbrack{\mu_b}
        \} 
        \right].
      $$

      It is easy to observe that
      \begin{align*}
        D_1 &= 
        \left[ 
          \{
            \bbrack{\lambda_1}, \ldots, \bbrack{\lambda_b}, 
            \bbrack{\mu_1} \ldots, \bbrack{\mu_b}
          \} / \{
            \bbrack{\lambda_b}, \ldots, \bbrack{\lambda_1}, 
            \bbrack{\mu_1}, \ldots, \bbrack{\mu_b}
          \}
         \right] \\
        & = \left|
          \begin{array}{ccc}
            0      & \cdots  & 1 \\
            \vdots & \rotatebox{45}{$\dots$} & \vdots \\
            1      & \cdots  & 0
          \end{array}
        \right| = (-1)^{b(b-1)/2},
      \end{align*}
      and that $D_2 = 1$, because it is the signature of a product of
      even permutations.

      As a conclusion, we have that 
      $$\Tor{\mathcal{H}^{(1)}_*}{\{\hbasisrel{E_L}\}}{\emptyset} =
      (-1)^{b(b-1)/2}.$$

    \item \emph{Computation of
        $\Tor{\mathcal{H}^{(0)}_*}{\{\hbasisrel{E_L}\}}{\emptyset} $}.

      The connecting homomorphism $\delta^{(0)}$ maps $\bbrack{\gamma_\ell}$ to
      $\bbrack{p_b} - \bbrack{p_\ell}$, for all $\ell = 1, \ldots, b-1$. 
      Also observe that $j_*^{(0)}(\bbrack{p_b}) = \bbrack{p_b}$. 
      Thus, the torsion of the short exact sequence
      $\mathcal{H}_*^{(0)}$ is given by the following base change
      determinant:
      \begin{align*}
        \Tor{\mathcal{H}^{(0)}_*}{\{\hbasisrel{E_L}\}}{\emptyset} 
        &= 
          \left[
            \{ 
              \delta(\bbrack{\gamma_1}), \ldots, \delta(\bbrack{\gamma_{b-1}}), 
               \bbrack{p_1}
            \} / \{
              \bbrack{p_1}, \ldots, \bbrack{p_b} 
            \}
            \right]\\
        &= \left[
          \begin{array}{cccc}
            -1  &  & & \\
            \vdots & \ddots &  &  \\
            & & -1 & \\
            1 & \cdots & 1& 1 
          \end{array}
        \right] =(-1)^{b-1}.
      \end{align*}
    \end{itemize}
  \end{proof}
  Using this result, we conclude that the torsion $\Tor{\mathcal{H}_*}{\{\hbasisrel{E_L}\}}{\emptyset}$ of $\mathcal{H}_*$
  is given by
  $\prod_{i=0}^2{\Tor{\mathcal{H}^{(i)}_*}{\{\hbasisrel{E_L}\}}{\emptyset}}^{(-1)^{i}} 
    = (-1)^{b(b-1)/2}.$
\end{proof}

\bibliographystyle{amsalpha}
\bibliography{ref}

\end{document}